\documentclass[12pt,twoside]{amsart}
\usepackage{amsfonts,amsmath}
\usepackage{enumerate}
\usepackage{mathtools}
\usepackage{xcolor,colortbl}
\usepackage{url}
\usepackage{hyperref,cleveref}
\usepackage{subcaption}
\usepackage{tikz}
\usetikzlibrary{decorations.markings}

\def\endpf{\relax\ifmmode\expandafter\endproofmath\else
  \unskip\nobreak\hfil\penalty50\hskip.75em\hbox{}\nobreak\hfil\bull
  {\parfillskip=0pt \finalhyphendemerits=0 \bigbreak}\fi}
\def\bull{\vbox{\hrule\hbox{\vrule\kern3pt\vbox{\kern6pt}\kern3pt\vrule}\hrule}}

\newtheorem{defn}{Definition}[section]
\newtheorem{lemma}[defn]{Lemma}

\newtheorem{proposition}[defn]{Proposition}

\newtheorem{conjecture}[defn]{Conjecture}

\newtheorem{question}[defn]{Question}

\newcommand{\zz}{{\mathbb Z}}
\newcommand{\rr}{{\mathbb R}}

\newcommand{\spin}{\ifmmode{\rm Spin}\else{${\rm spin}$\ }\fi}
\newcommand{\spinc}{\ifmmode{{\rm Spin}^c}\else{${\rm spin}^c$\ }\fi}

\newcommand{\tors}{{\it Tors}}

\newcommand{\calc}{\mathcal{C}}

\definecolor{Gray}{gray}{0.8}

\newif\ifpic
\picfalse    

\begin{document}

\title{Knots and 4-manifolds}
\author[Brendan Owens]{Brendan Owens}
\address{School of Mathematics and Statistics \newline\indent 
University of Glasgow \newline\indent 
Glasgow, G12 8SQ, United Kingdom}
\email{brendan.owens@glasgow.ac.uk}
\date{\today}
\thanks{}

\begin{abstract}  These notes are based on the lectures given by the author during Winter Braids IX in Reims in March 2019.  We discuss slice knots and why they are interesting, as well as some ways to decide if a given knot is or is not slice.  We describe various methods for drawing diagrams of double branched covers of knots in the 3-sphere and surfaces in the 4-ball, and how these can be useful to decide if an alternating knot is slice.  We include a description of the computer search for slice alternating knots due to the author and Frank Swenton.
\end{abstract}

\maketitle

\pagestyle{myheadings}
\markboth{BRENDAN OWENS}{WINTERBRAIDS IX: KNOTS AND 4-MANIFOLDS}


These notes are a written accompaniment to the series of three lectures given by the author during the Winter Braids IX school and conference held in Reims, France in March 2019.  The goal of the lectures was to give an inviting introduction to 4-dimensional aspects of classical knot theory, with a focus on slice knots and some ways to find or obstruct them.  We prioritise intuition over precision, and urge the interested reader to consult other sources for further reading and detail.  We aim to provide such sources as we go along.  

There are many good introductory books on knot theory.  Three that are particularly useful for the material in these notes are those by Lickorish, Livingston, and Rolfsen \cite{lick,liv,rolf}.  An excellent introduction to smooth 4-dimensional topology is the book of Gompf and Stipsicz \cite{GS}.


\section{Lecture 1}

\subsection*{Slice knots} A knot is a smooth embedding of $S^1$ in $S^3$, considered up to smooth isotopy, and sometimes also up to reversal.  Some examples are shown below, drawn in the usual way as projections with crossing information at the double points.  If our knots are oriented we indicate this with an arrow, as shown.
\begin{center}
\includegraphics[width=12cm]{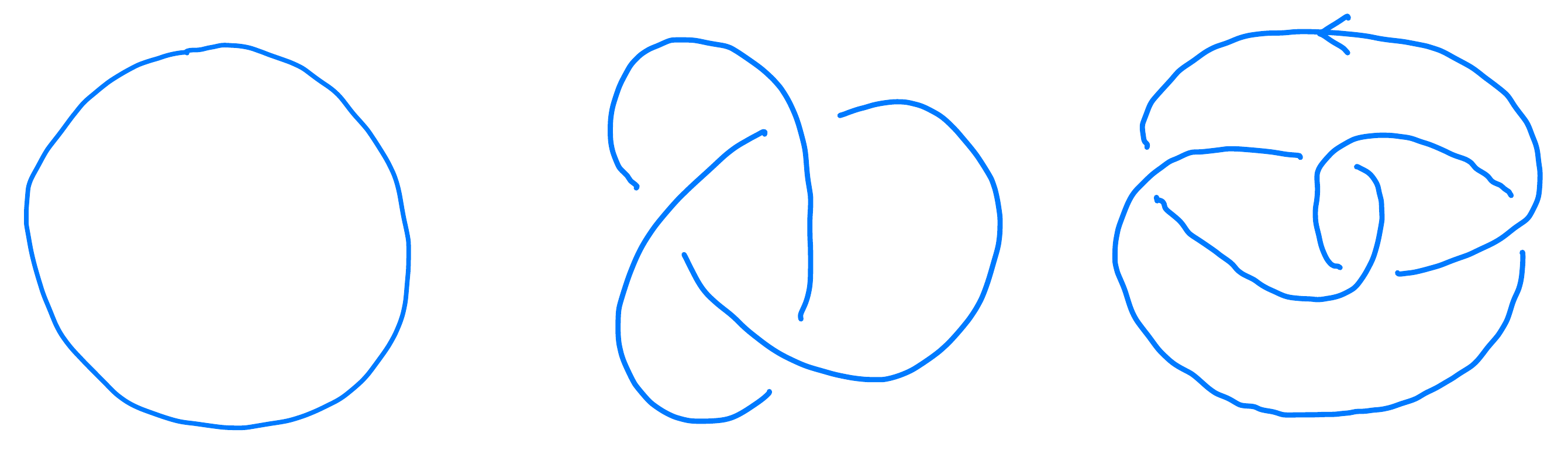} 
\end{center} 
The above diagrams are \emph{alternating}: as you traverse the knot, you alternate between over- and under-crossings.  A knot is called alternating if it admits an alternating diagram.

A 2-knot is an embedding of $S^2$ in $S^4$, again up to isotopy.  This can be either smooth or topologically locally flat, the latter meaning that at each point of the 2-knot, one can find continuous local coordinates in which the 2-knot maps to a coordinate plane \cite{FNOP}.
A knot $K$ is \emph{slice} if $K$ is the intersection of a 2-knot with the equatorial $S^3$ in $S^4$, as indicated in the cartoon below.  We say $K$ is smoothly or topologically slice, according to whether the 2-knot is smooth or just locally flat.
\begin{center}
\includegraphics[width=8cm]{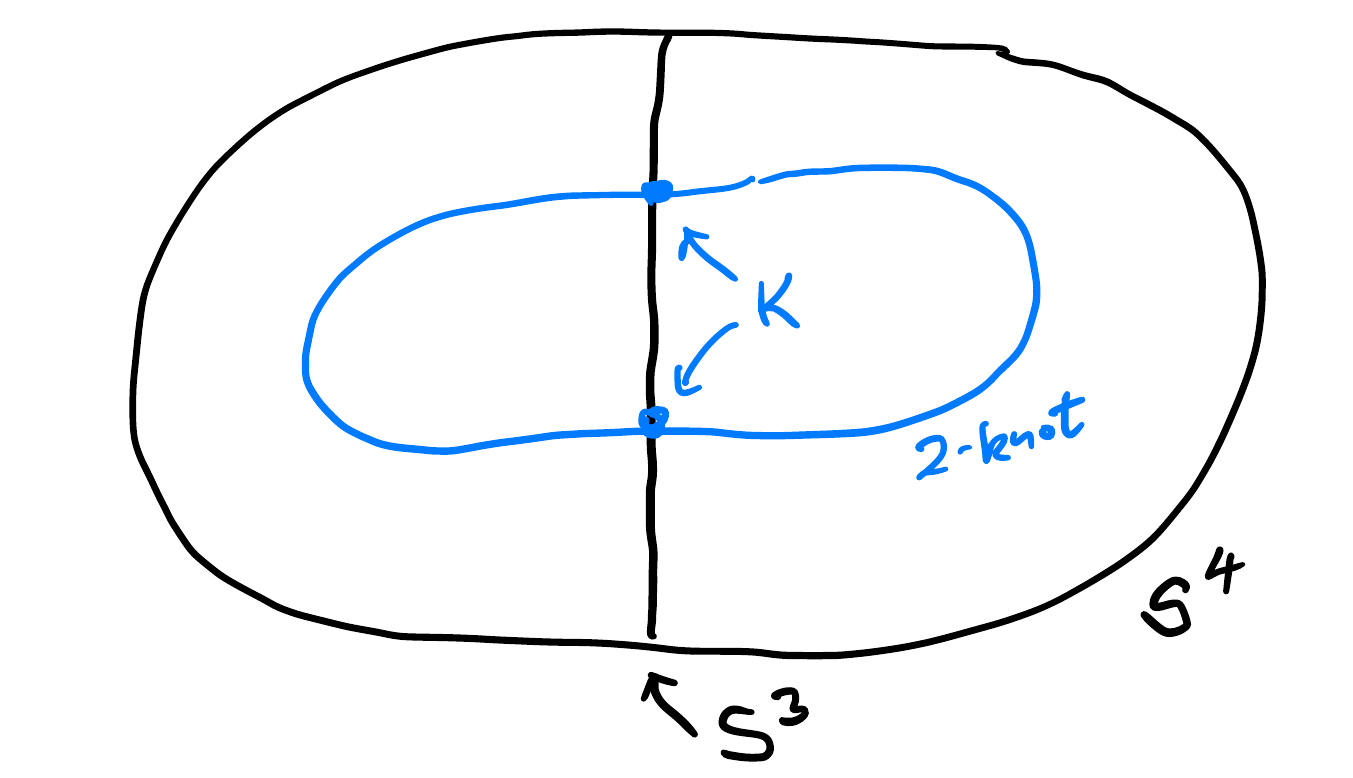} 
\end{center} 

Equivalently $K$ is the boundary of an embedded $D^2$ in $D^4$.  Perhaps surprisingly, this simple concept has fundamental significance in 4-dimensional topology.  We will give some examples, and then briefly describe two applications of slice knots.

\subsection*{Slice and ribbon surfaces in the 4-ball}
It turns out that surfaces in the 4-ball can be represented and studied using knot and link diagrams.  We begin with a trivial example: the unknot bounds an embedded disk in $S^3$.  Thinking of $S^3$ as the boundary of the 4-ball, this also gives a disk in $D^4$; we modify it to obtain a properly embedded disk, with boundary mapping to boundary and interior to interior, by pushing points in the interior of the disk into the interior of the 4-ball.
Two nontrivial slice knots, the stevedore knot and the square knot, are shown below.
\begin{center}
\includegraphics[width=12cm]{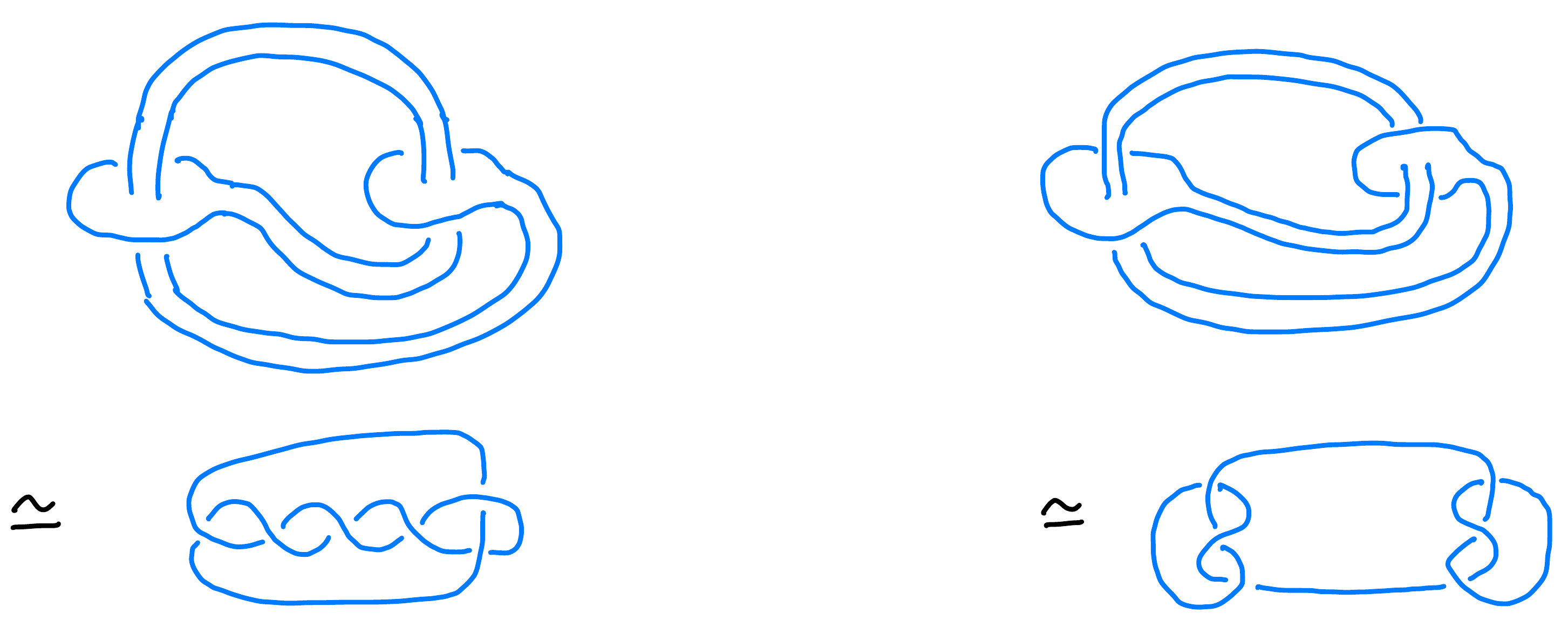} 
\end{center} 
Consider the top left diagram.  One can view this as a narrow strip or band connecting two disks in $S^3$.  The band passes through each of the disks once, making what is called a ribbon singularity: this is a self-intersection of a surface along an interval, such that the preimage of the interval consists of one internal interval in the domain surface and one properly embedded interval, with its endpoints on the boundary.  Thus we see an immersed disk in $S^3$ bounded by the stevedore knot.  As with the trivial disk bounded by the unknot, we can push the interior of this disk into the interior of the 4-ball.  Doing this carefully enables us to resolve the ribbon intersections and obtain an embedded disk.  We simply make sure that the part of the disk containing the internal interval gets pushed further into the 4-ball than the part of the disk containing the properly embedded interval.  Thus the pictures above demonstrate that the stevedore and square knots are slice.  A surface in $S^3$ is called \emph{ribbon-immersed} if it is  embedded except for some ribbon singularities as in these examples.  We see that in general the boundary of a ribbon-immersed disk is a slice knot.

It is helpful to view these disks from the point of view of Morse theory (we will give a brief introduction to Morse theory and handlebodies in Lecture 2).   Given any smooth surface $F$ in the 4-ball, we may isotope it so that the radial distance function restricts to give a Morse function on $F$.  Then we can draw the intersection of $F$ with 3-spheres of fixed distance from the origin; the result is a sequence of link diagrams in $S^3$ called a ``movie".  The movie version of the same slice disk for the stevedore is shown below:
\begin{center}
\includegraphics[width=12cm]{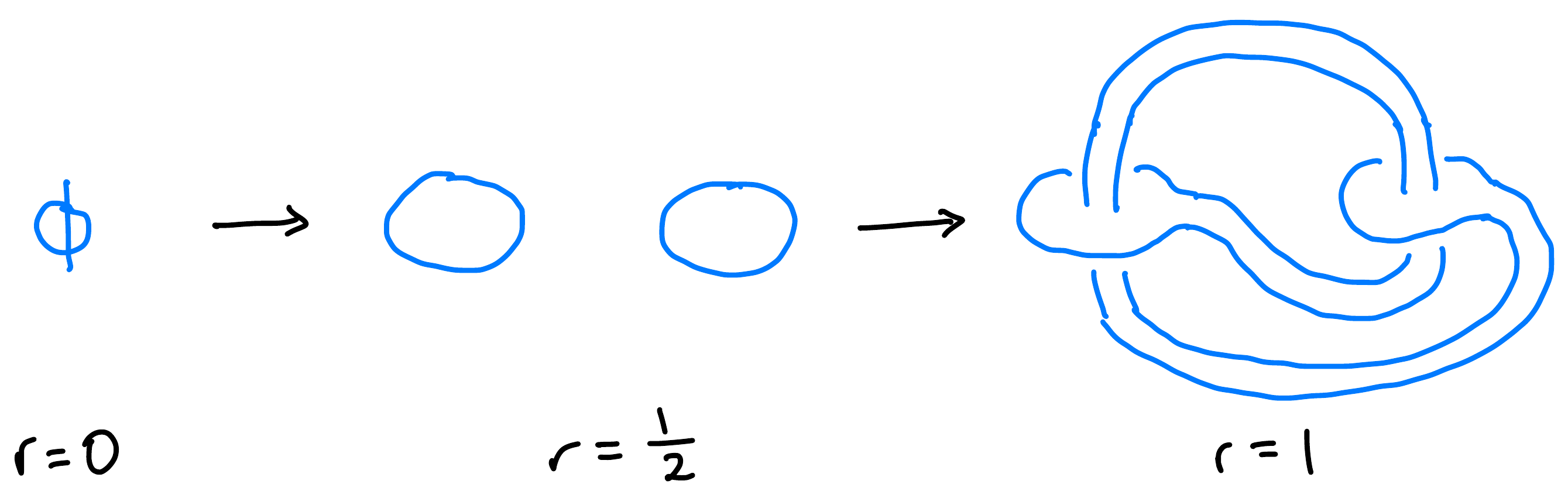} 
\end{center} 
What are the steps in this sequence?  From $r=0$ to $r=1/2$, we have passed two minima of $r|_F$, resulting in a two component unlink.  In general a minimum gives rise to an extra unknot component in the next frame of the movie.  Correspondingly, if we pass a maximum, then an unknot, which is contained in a 3-ball disjoint from the rest of the link, disappears.
What happens between $r=1/2$ and $r=1$?  On the level of diagrams, we perform a \emph{band move}: that is to say, we find an embedded rectangle in $S^3$ which intersects the link shown in the $r=1/2$ picture on two opposite sides, and then we erase those two sides of the rectangle and replace them by the other two sides.  In terms of Morse theory, we pass a saddle point.  To see this it may help to look at the following picture which might occur at $r=0.9$: an isotopy has occurred between $r=1/2$ and $r=0.9$, and we can now see that the $r=1$ frame results from this by passing a saddle point.
\begin{center}
\includegraphics[width=5cm]{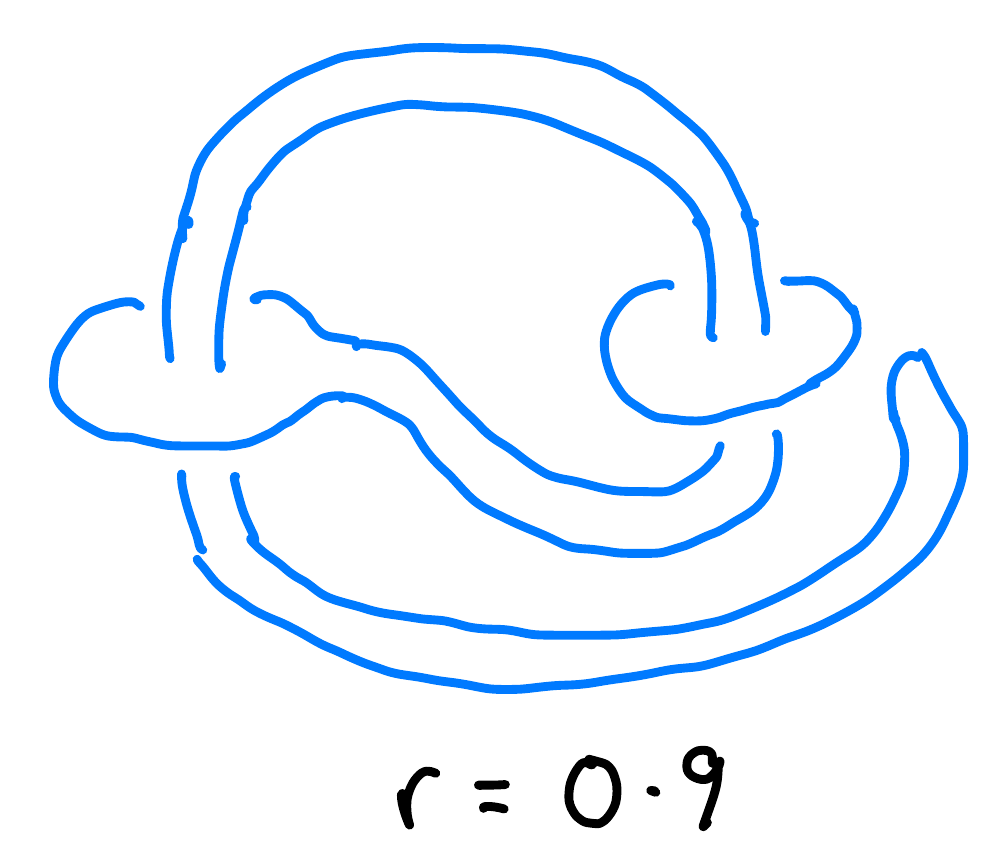} 
\end{center} 

An embedded surface $F$ in $B^4$ is called \emph{ribbon} if the radial distance function restricts to be Morse with no maxima.  From the preceding example, we see that this gives rise to a ribbon-immersed surface in $S^3$, and that by pushing the interior into the 4-ball to resolve the singularities we may recover $F$.

As usual one can turn such handlebody decompositions upside down.  In this context this simply means we run the movie backwards.  Now maxima  result in adding split disjoint unknots to our movie frame, and minima correspond to removing them.  (The split union of two links  in $S^3$ is the link which contains the given links in disjoint 3-balls.)  Saddles still correspond to band moves: the operation of replacing two sides of an embedded rectangle with the other two sides is its own inverse.  Also note that one may push minima deep into the 4-ball and maxima close to the surface, so that the corresponding movie events are sorted with minima followed by saddles followed by maxima, or the reverse.  Thus every smooth compact surface in $B^4$ bounded by a given link $L$ is given by a movie, which in turn is given by a finite set of band moves applied to a split union of $L$ with an unlink, which results in another unlink.  Here the first unlink corresponds to maxima and the second to minima.

We say a knot is \emph{ribbon} if it bounds a ribbon disk in $B^4$, or equivalently if it bounds a ribbon-immersed disk in $S^3$.  Since the Euler characteristic of a disk is one, being ribbon is equivalent to the existence of $k$ band moves applied to the knot which convert it to the $(k+1)$-component unlink.  More generally, a knot $K$ is (smoothly) slice if and only if there exists a set of $k$ band moves applied to the split union of $K$ and an $l$-component unlink, converting it to the $(k-l+1)$-component unlink.  A longstanding question of Fox asks if every slice knot is ribbon \cite{fox}.

\subsection*{Exotic $\rr^4$ from Khovanov homology} From the work of many mathematicians, including Rad\'{o},  
Moise, and Stallings, we know that in any dimension but 4, there is a unique smooth structure on $\rr^n$ \cite{moise,rado,stallings}.  It is an amazing fact that $\rr^4$ admits more than one smooth structure.  In fact it follows from combining deep work of Donaldson, Freedman, Gompf, and Taubes \cite{taubes} that in fact $\rr^4$ admits uncountably many smooth structures!  This requires gauge theory and in particular a great deal of analysis.  We will give a brief sketch here of a gauge-theory-free proof due to Rasmussen, following a suggestion of Gompf, in which the existence of an exotic smooth structure on $\rr^4$  follows from the existence of a knot which is topologically slice but \emph{not} smoothly slice.  Examples are the $(-3,5,7)$ pretzel knot, shown below, and the untwisted Whitehead double of the trefoil \cite[Figure 6.13]{GS}.  In both cases topological sliceness follows from a deep theorem of Freedman \cite{freedman}, which tells us that knots with trivial Alexander polynomial are topologically slice, while smooth nonsliceness follows from a computation of Rasmussen's $s$-invariant, defined using Khovanov homology \cite{ras}.
\begin{center}
\includegraphics[width=8cm]{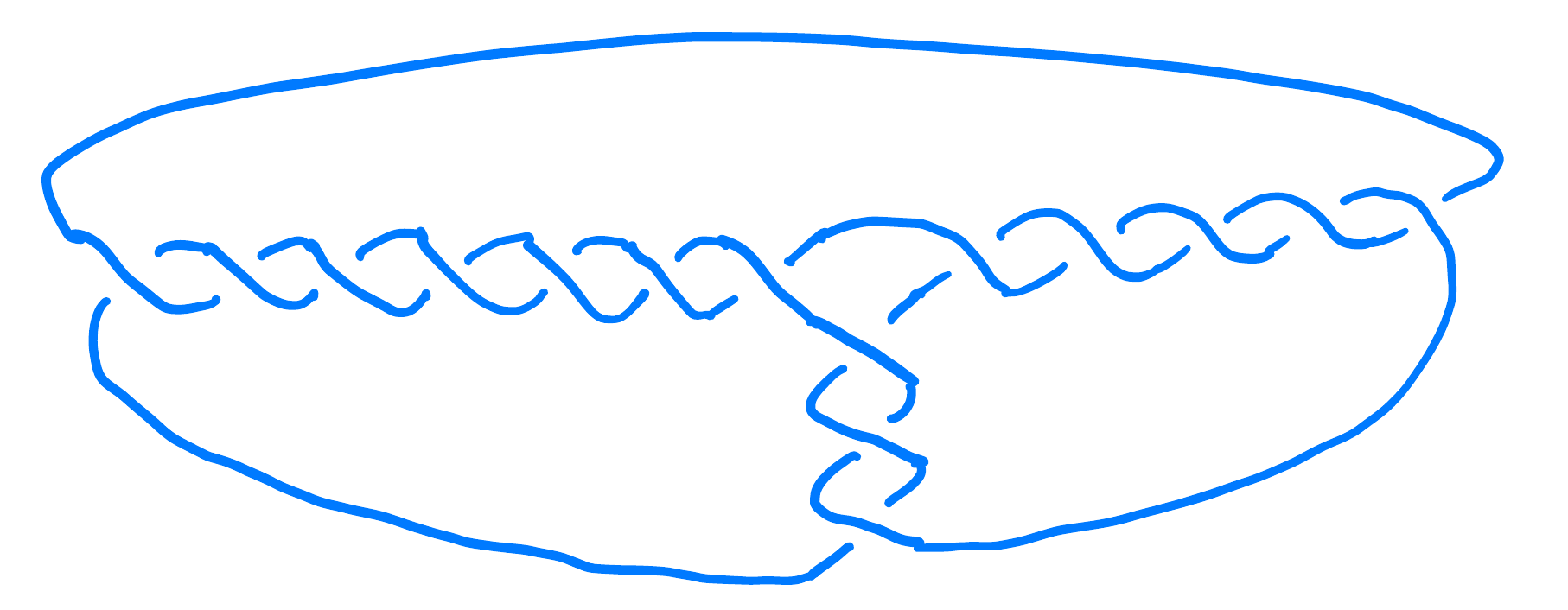} 
\end{center} 

Given a knot $K$ in $S^3$ we obtain a 4-manifold with boundary called the \emph{knot trace} $X_K$; this is obtained by gluing a 2-handle to the 4-ball along $K$.  To be slightly more precise, we have
$$X_K=D^4\cup (D^2\times D^2),$$
where we quotient by a gluing map defined on $\partial D^2\times D^2$ which takes $\partial D^2\times 0$ to $K$, and $\partial D^2\times 1$ to a 0-framing, or if you prefer to a push-off or longitude of $K$ which is nullhomologous in $S^3\setminus K$, or equivalently which lies on a Seifert surface for $K$.  Cartoon below:
\begin{center}
\includegraphics[width=8cm]{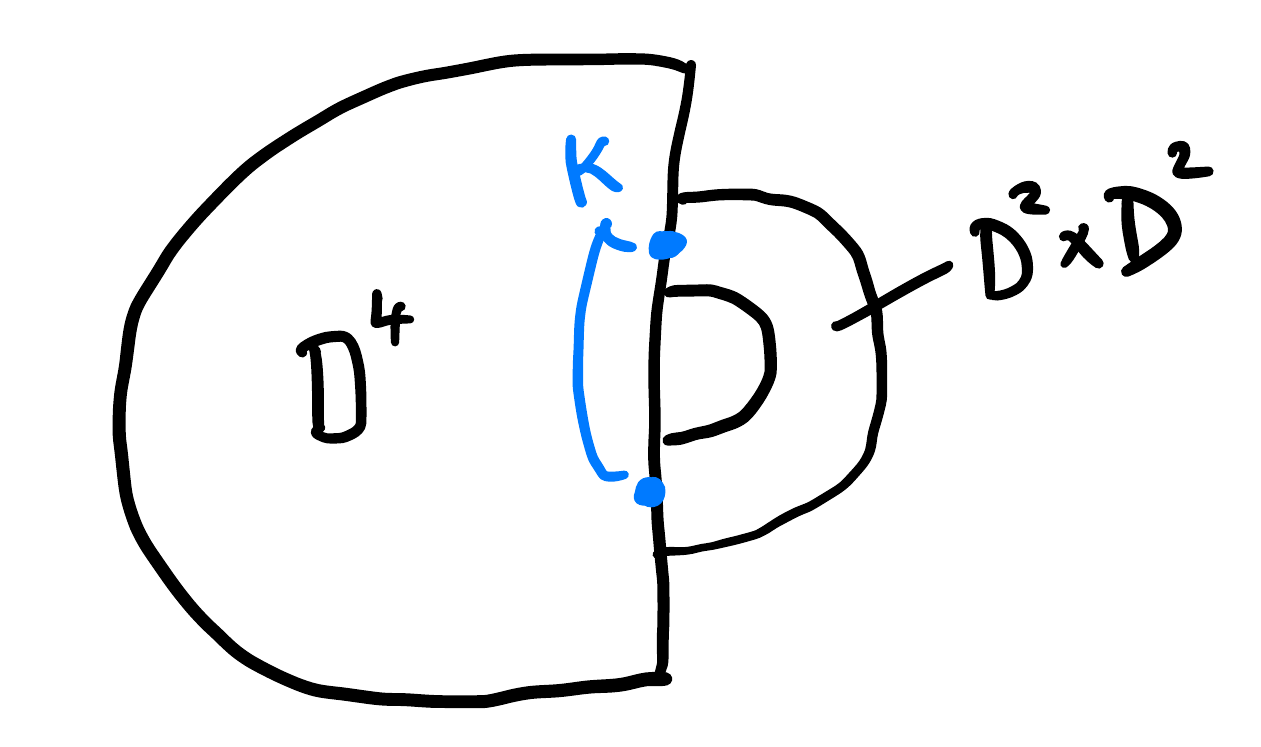} 
\end{center} 

Then we have the following lemma, which holds in either the smooth or topological category:
\begin{lemma}[Kirby-Melvin]
\label{lem:KM}
$K$ is slice $\iff X_K\hookrightarrow S^4\iff X_K\hookrightarrow \rr^4$.
\end{lemma}
A literal sketch of the proof is as follows:
\begin{center}
\includegraphics[width=12cm]{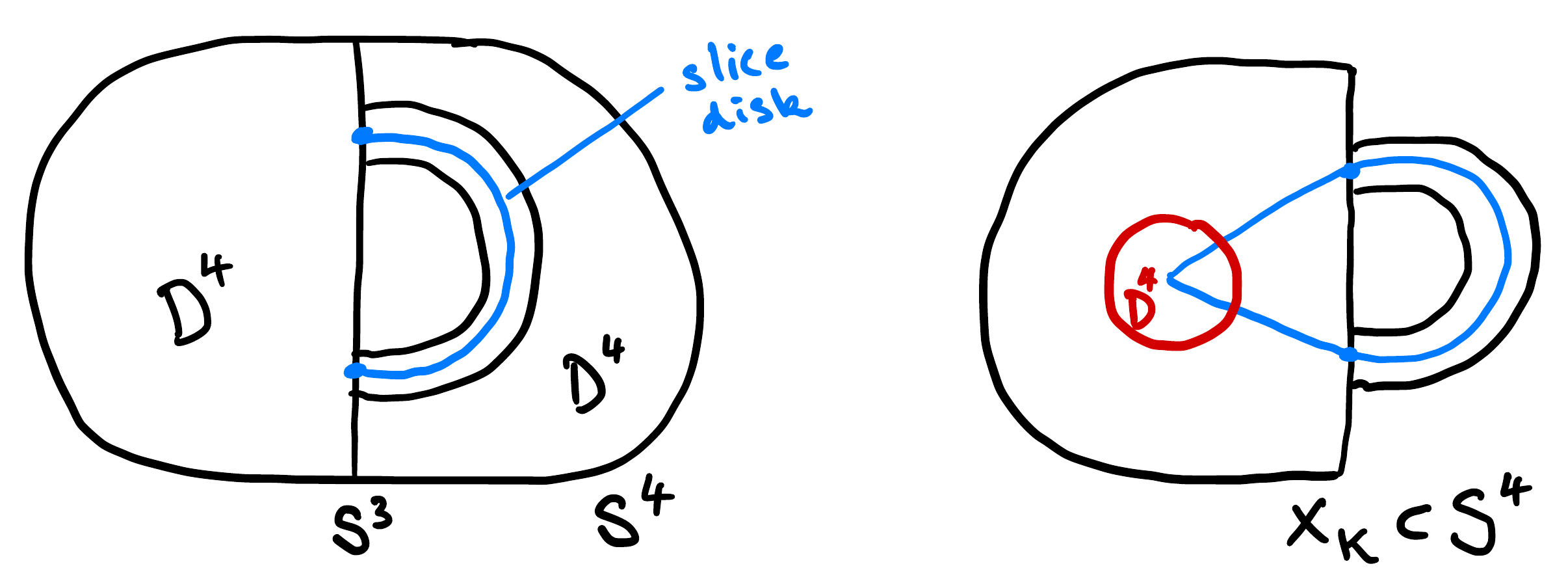} 
\end{center} 
The cartoon on the left gives the forward direction, and the converse comes from the sketch on the right.  The smaller $D^4$ inside $X_K$ is standard (it is a ball neighbourhood of the centre of $D^4\subset X_K$) so its complement in $S^4$ is a ball in which we see a slice disk.  For details see Kirby-Melvin or Miller-Piccirillo \cite{kirmel,millerpic}.

Finally we take a knot $K$ as above which is topologically but not smoothly slice.  Thus $X_K$ embeds topologically in $\rr^4$.  Freedman-Quinn tell us that $\rr^4\setminus X_K$ admits a smooth structure, and this can be extended over all of $\rr^4$, with $X_K$ as a smooth submanifold \cite{FQ}.  Since Kirby-Melvin tell us that $X_K$ does not embed smoothly in standard $\rr^4$, this smooth structure must be exotic.

As a final note about the above, we must mention the proof by Piccirillo that the Conway knot is not slice, which made a similar use of \Cref{lem:KM} and Rasmussen's invariant \cite{hom2,pic}.  She exhibited a knot $K$ with nonzero Rasmussen invariant which is therefore not slice, and showed that the knot trace of $K$ is diffeomorphic to that of the Conway knot.

\subsection*{The concordance group of knots}
The set of oriented knots in $S^3$ is an abelian monoid with the operation given by connected sum.  This can be realised diagrammatically as shown below, or one can excise two standard pairs $(D^3,D^1)$ from each of $(S^3,K_1)$ and $(S^3, K_2)$ and glue the complements together matching orientations.
\begin{center}
\includegraphics[width=12cm]{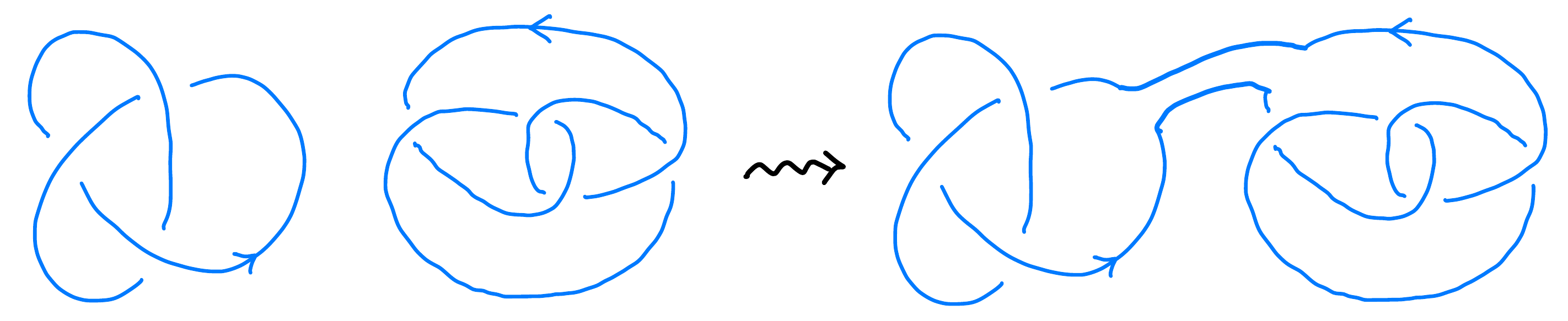} 
\end{center} 

Sadly there are no inverses for this operation at the level of knots and isotopy: knot genus is additive by \cite{schub} (or see \cite[Theorem 2.4]{lick}), so the only way to get an unknot by connected sum is if each summand is the unknot.

An isotopy of knots $K\simeq K'$ is equivalent to an embedding 
$$S^1\times I \hookrightarrow S^3\times I$$
which gives $K$ and $K'$ on the boundary, and which is the identity on the interval summand.  We can remove this restriction and allow any smooth (or locally flat) proper embedding of the annulus into $S^3\times I$ which cobounds $K$ and $K'$.  Such an embedded annulus is called a \emph{concordance} between $K$ and $K'$.  It is quite straightforward to see that this gives an equivalence relation on the set of oriented knots, which we will write $K\sim K'$.

It is also straightforward to see that a knot $K$ is slice if and only if $K$ is concordant to the unknot.  More generally, $K\sim K'$ if and only if $K\#-K'$ is slice, where $-K$ denotes the mirror reverse of the oriented knot $K$.  Recall that the mirror of a knot in $S^3$ is its image under reflection in a plane; thus one obtains a diagram for $-K$ from one for $K$ by changing all crossings and reversing the orientation.  If $K\sim K'$, then drilling out a regular neighbourhood of the image in $S^3\times I$ of an arc $\{1\}\times I$ from the concordance $S^1\times I$ results in a slice disk in $D^4$ bounded by $K\#-K'$.  Conversely, if $K\#-K'$ is slice then we can add a $(3,1)$ handle pair to $(D^4,K\#-K')$ to obtain a concordance.  Breaking that down: a slice disk for $K\#-K'$ gives rise to an annulus in $D^4$ bounded by the disjoint union of $K$ and $-K'$ by applying a band move to undo the connected sum.  We then attach a 3-handle to the 4-ball along a 2-sphere separating $K$ and $-K'$ to obtain the required embedded annulus in $S^3\times I$.

One then has that $\calc=\{\mbox{oriented knots}\}/\!\sim$ is a group with operation given by connected sum, and with the inverse of $K$ represented by $-K$.  There are two versions of this, depending on whether the concordances are smoothly or locally flatly embedded.  For further reading on knot concordance, please see for example \cite{hom,liv,liv2}.

One can also consider concordance of links, with an equivalence relation given by embeddings of disjoint unions of annuli in $S^3\times I$.  This does not give rise to a group in the same way for links of more than one component.  For various approaches to defining concordance groups of links, see \cite{do,heddenkuzbary,hosokawa}.

{\bf Basic problem:} how to decide if a given knot $K$ is or is not slice?

Showing $K$ is slice:
\begin{itemize}
\item exhibit a slice or ribbon disk as described above
\item exhibit a handle decomposition of the 4-ball involving $0$-, $1$-, and $2$-handles.  Then the boundaries of the cocores of the 2-handles are slice knots in $S^3$, and one can apply Kirby calculus to obtain diagrams of them in the standard picture of $S^3$.  This method has been successfully used by Gompf-Scharlemann-Thompson and also Abe-Tange \cite{AT,GST}.
\end{itemize}

To show $K$ is nonslice, there are a plethora of obstructions, some of which obstruct topological sliceness and some just smooth sliceness.  These include
\begin{itemize}
\item determinant (square for slice knots)
\item signature and Levine-Tristram signatures
\item Fox-Milnor condition on the Alexander polynomial
\item Herald-Kirk-Livingston condition on twisted Alexander polynomials
\item Heegaard Floer concordance invariants
\item Rasmussen's invariant and other invariants coming from variants and refinements of Khovanov homology.
\end{itemize}

Some sources for the above include \cite{hkl,hom,LL1,LL2,LS,liv2,ras}.

Some notable progress in recent years includes Lisca's classification of slice two-bridge knots, together with his verification that all such are ribbon \cite{lisca}.  The obstruction used by Lisca is an application of Donaldson's diagonalisation theorem, and double branched covers of the 4-ball; we will look at this in more detail in Lecture 3.  Subsequently, Lecuona applied Lisca's method to prove that slice implies ribbon for a family of alternating Montesinos knots \cite{lecuona}.

{\bf Sample problems:}
\begin{itemize}
\item Classify slice knots in some chosen family of knots
\item Classify ribbon knots in some chosen family of knots
\item Prove slice implies ribbon for some chosen family of knots
\item Classify slice disks (or ribbon disks, or ribbon disks with up to a certain number of critical points) for some chosen knot.
\end{itemize}

The following is the most basic example of  the last question above.
\begin{question}
\label{q:unknot}
Is there only one isotopy class of ribbon disk for the unknot?
\end{question}
A positive answer  for ribbon disks with at most three critical points follows from theorems of Gabai and Scharlemann \cite{gabai,schar}.  These show that the only way to apply a band move to the unknot and obtain the $2$-component unlink is the obvious planar diagram, which gives rise to the standard disk pushed in from $S^3$.

For other work on classification of slice disks, see \cite{CP,JZ,MZ}.

Our focus for the rest of these lectures will be on alternating knots, and whether we can determine if they are slice.  The powerful obstruction from Donaldson's diagonalisation theorem and double branched covers used by Lisca extends to this case, as we will describe.  We will see that in addition to giving us an obstruction, Donaldson's theorem actually helps us to find ribbon disks.


\section{Lecture 2}
The main goal of this lecture is to look at double branched covers of knots and links in the 3-sphere and surfaces in the 4-ball.  We begin with a very brief introduction to Morse theory and Kirby calculus, which enables us to draw pictures of smooth 3- and 4-dimensional manifolds.  For more details we refer the reader to \cite{akbulut,GS,milnorMT}.

Let $M^n$ be a smooth compact manifold, possibly with boundary.  We may choose a Morse function
$$f:M\to \rr.$$
This has the property that at each critical point $p$ of $f$, there exist local coordinates in which
$$f=-x_1^2-\dots-x_k^2+x_{k+1}^2+\dots+x_n^2.$$
We say that the index $i(p)$ of $p$ is $k$.  This leads to a \emph{handlebody decomposition} of $M$ as follows.   We consider sublevel sets $M_b=f^{-1}(-\infty,b]$.  These are the same up to diffeomorphism for nearby regular values of $b$, and change when the level $b$ passes a critical point of index $k$ by attachment of a $k$-handle, which is $D^k\times D^{n-k}$ attached to $M_b$ along $\partial D^k\times D^{n-k}$.  This is illustrated for the most standard example (the height function on a torus) below.
\begin{center}
\includegraphics[width=12cm]{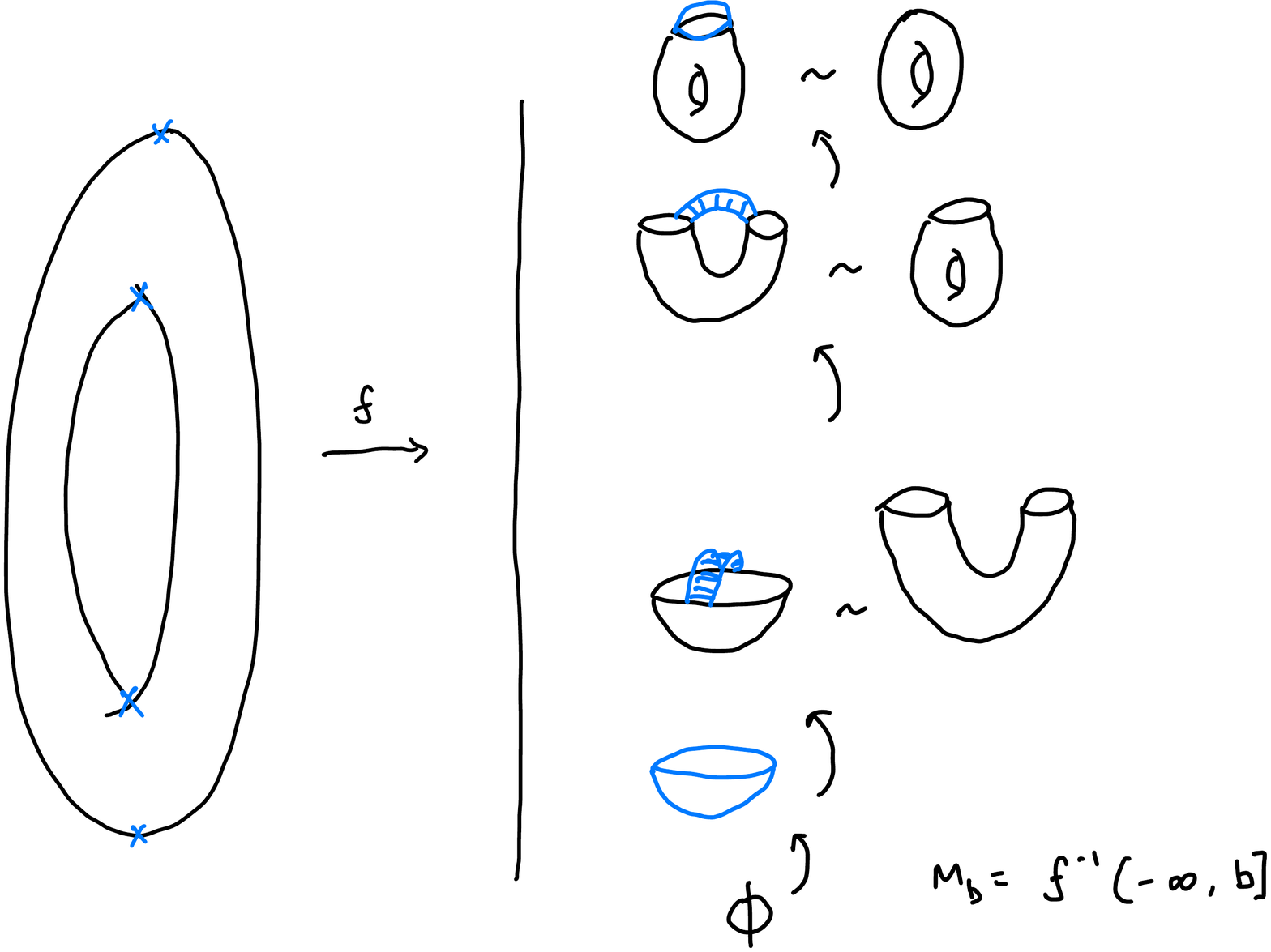} 
\end{center} 
There are 4 critical points, of index 0, 1, 1, and 2.  The first sublevel set is the empty set, to which we attach a 0-handle $D^2$ when we pass the minimum of the height function $f$.  Then each time we pass an index one critical point we attach a band $D^1\times D^1$ along two opposite faces of the rectangle.  Finally we pass a maximum and attach a 2-handle, which in this dimension is just a 2-disk attached along its boundary circle.

A huge benefit of this description for low-dimensional topologists is that it enables us to draw $(n-1)$-dimensional pictures of $n$-dimensional manifolds.  Let's see this for the above example.  There is a single minimum point, so there is a single $0$-handle.  The boundary of the $0$-handle is a copy of $S^1$, and this is where we draw our diagram of the torus.  We just need to specify how the two 1-handles are attached; we will then have built a surface with a single boundary component, and there is then just one manifold that we can obtain by attaching a 2-handle, so we don't need to specify any more information except that there is a 2-handle to be attached.  This results in the following picture of the torus.
\begin{center}
\includegraphics[width=\textwidth]{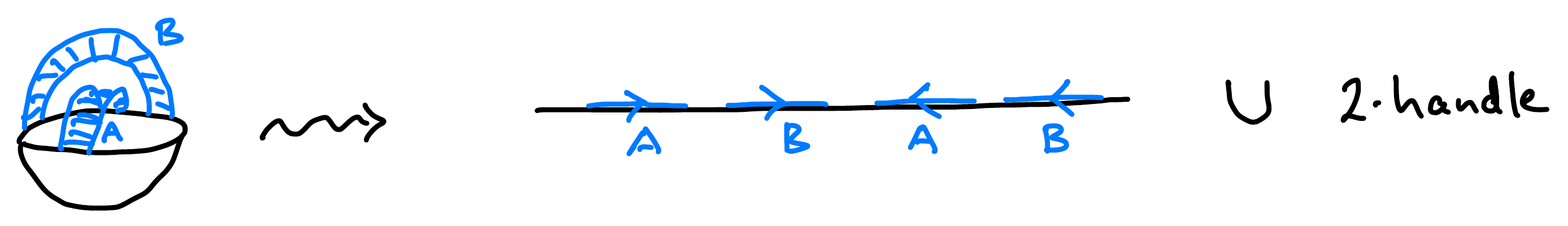} 
\end{center} 
In a similar way we can obtain 3-dimensional diagrams for 4-manifolds, which is of great benefit for those of us who have difficulty seeing in 4 dimensions!  These are commonly referred to as Kirby diagrams.  Suppose that $M$ is a compact connected oriented 4-manifold.  A suitable Morse function gives rise to a handle decomposition
$$M^4=0\mbox{-handle}\cup\{1\mbox{-handles}\}\cup\{2\mbox{-handles}\}\cup\{3\mbox{-handles}\}(\cup4\mbox{-handle}?),$$
where there is a 4-handle attached if the manifold is closed.  We draw the attaching regions in $\partial(0\mbox{-handle})=S^3$.  The attaching region for a 1-handle is $\partial D^1\times D^3$, which is a pair of disjoint 3-balls.  The attaching region for a 2-handle is $\partial D^2\times D^2$ which is a solid torus.  We specify this by giving a knot in $S^3$ and a framing coefficient which is an integer; the $0$-framing of a knot in $S^3$ is equal to a push-off on its Seifert surface, or in other words a push-off of the knot which is nullhomologous in the knot complement.  More generally the $m$-framing is a push-off of the knot which has linking number $m$ with the knot.  Here are some examples.
\begin{center}
\includegraphics[width=\textwidth]{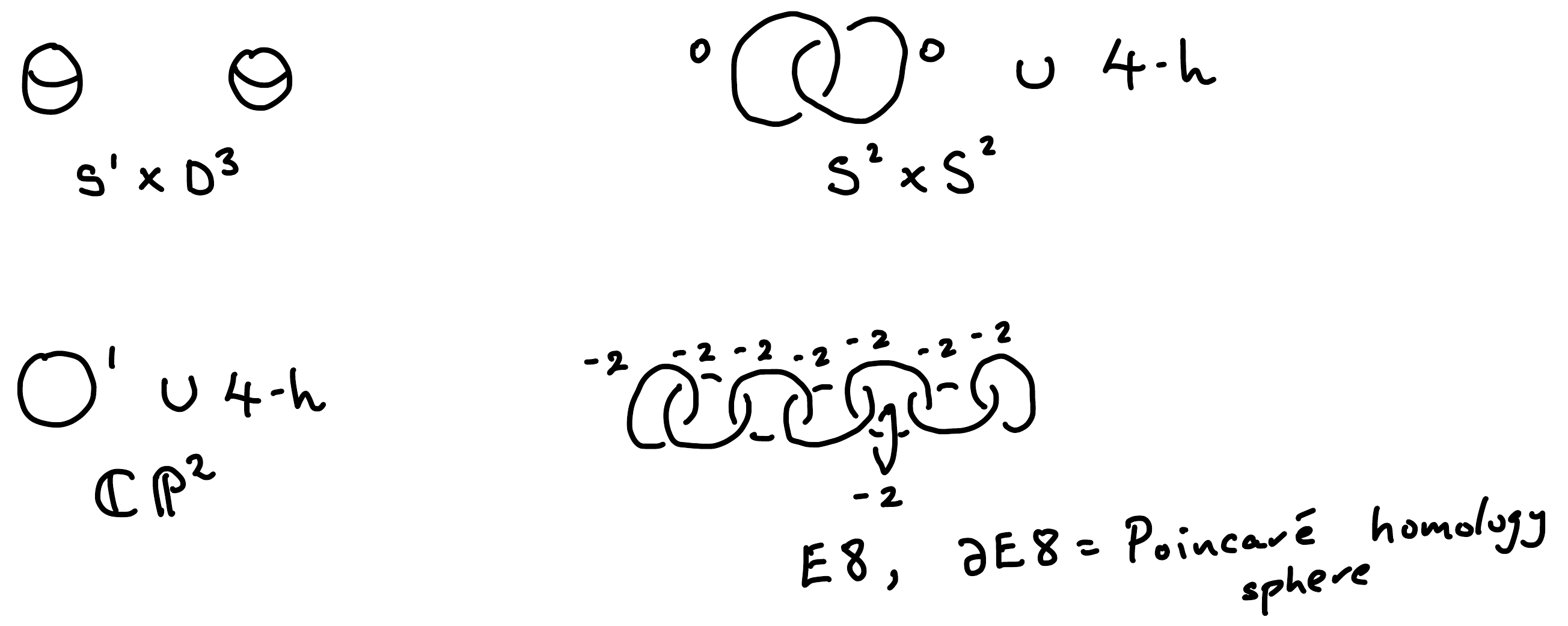} 
\end{center} 
A result of Laudenbach and Po\'{e}naru \cite{LP} tells us that any self-diffeomorphism of a connected sum of copies of $S^1 \times S^2$ extends over the boundary sum of copies of $S^1\times D^3$.  This means that for a closed 4-manifold, we do not need to specify how the 3- and 4-handles are attached.  However, for a 4-manifold with boundary,  it is necessary to specify the embedded 2-spheres along which we attach any 3-handles.

There is a useful alternative way to draw 1-handles in Kirby diagrams, namely as dotted circles: one draws an unknot with a dot along it which is a meridian of an arc in the boundary of the 0-handle connecting the two balls which are the attaching region of the 1-handle.  Attaching circles of 2-handles which go over the 1-handle are instead drawn as passing through the dotted circle.  The following local picture gives a good idea of how to translate between these two methods of drawing Kirby diagrams.
\begin{center}
\includegraphics[width=6cm]{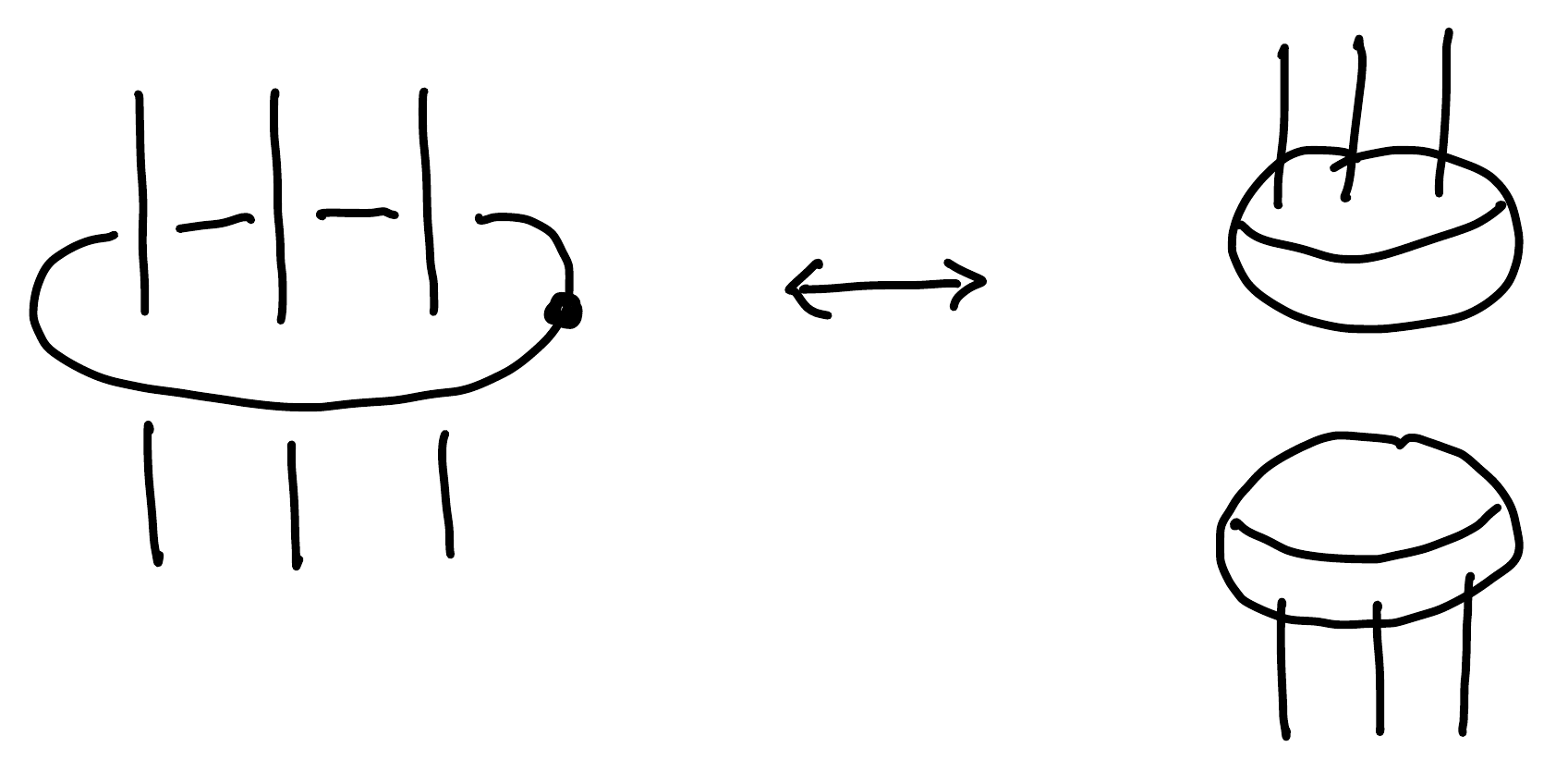} 
\end{center} 

It is straightforward to calculate the homology groups and the intersection pairing $Q_M$ on the second homology (Poincar\'{e} dual to the cup product pairing) from a Kirby diagram of $M$.  Kirby calculus is the application of certain moves such as handleslides which are known to preserve $M$ and/or its boundary 3-manifold.  For more details on Kirby diagrams, see \cite{akbulut,GS}.

There are two common and useful approaches for drawing pictures of 3-manifolds.  One is to appeal to a theorem of Rohlin \cite{rohlin} which tells us that every (closed oriented) 3-manifold is the boundary of a smooth 4-manifold, and we can just draw a Kirby diagram of such a bounding 4-manifold.  Another is to take a 3-dimensional handle decomposition and draw our picture in the oriented surface $\Sigma$ which is the boundary of the handlebody given by the union of the $0$- and $1$-handles, which is connected if there is a single $0$-handle.  We then draw two sets of circles on $\Sigma$: one set, commonly drawn in blue, which shows where to attach the $2$-handles, and one set in red which if cut along would undo the attachment of the $1$-handles.  These are called Heegaard diagrams.  Here are some simple examples.
\begin{center}
\includegraphics[width=\textwidth]{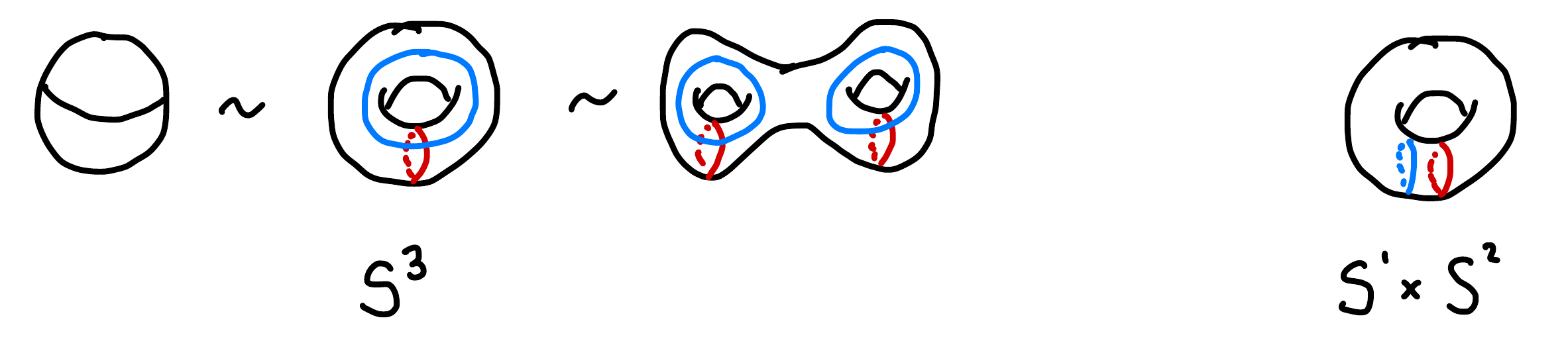} 
\end{center} 

\subsection*{Double branched covers}
Let $L$ be a link of $l$ components in $S^3$.  Denote a regular neighbourhood of $L$ by $\nu(L)\cong L\times D^2$.  Let $X$ be the link complement
$$X=S^3\setminus\overset{\circ}{\nu(L)}.$$
Then $H_1(X)\cong\zz^l$, generated by meridians $\mu_1,\dots,\mu_l$.  Let  $\widetilde{X}\rightarrow X$ be the double cover of $X$ which is nontrivial on each meridian.  The double cover $\Sigma(S^3,L)$ of $S^3$ branched along $L$ is then given as follows.
$$\begin{tikzpicture}[scale=0.5]
\node at (0,0) {$\Sigma(S^3,L)$};
\node at (3.25,0) {$=$};
\node at (6,0) {$\nu(L)$};
\node at (8.25,0) {$\cup$};
\node at (10,0) {$\widetilde{X}$};
\node at (0,-3) {$S^3$};
\node at (3.25,-3) {$=$};
\node at (6,-3) {$\nu(L)$};
\node at (8.25,-3) {$\cup$};
\node at (10,-3) {$X$};
\draw[->] (0,-0.6) to (0,-2.4);
\draw[->] (6,-0.6) to (6,-2.4);
\draw[->] (10,-0.6) to (10,-2.4);
\node at (6.5,-1.5) {\tiny $z^2$};
\end{tikzpicture}$$
In the same way we obtain the double branched cover $\Sigma(D^4,F)$ of a surface $F$, smoothly properly embedded in the 4-ball.  The surface $F$ does not have to be connected or orientable, and it may contain closed components.  We would like to learn to draw Kirby diagrams for these; this will be described in more detail in ongoing work of the author with Sa\v{s}o Strle \cite{GLslice}, and other accounts are given in \cite{akbulut, AK, GS}.  We follow a method due to Akbulut-Kirby, following a construction of Kauffman, also used by Gordon-Litherland \cite{AK,GL,kauf}.  This applies to the special case of branch locus $F^{n-2}$ in $D^n$ arising from an embedded submanifold in $\partial D^n$ by pushing the interior of the submanifold into the interior of the $n$-ball.

We cut open the ball along the trace of the homotopy that pushed the interior of the submanifold into the interior of $D^n$ to obtain the properly-embedded $F$.  This results in a copy of $D^n$ again, with a regular neighbourhood $N$ of $F$, which is the total space of an $I$-bundle over $F$, in its boundary.  This cut-open $D^n$ has corners, as illustrated below in the case $n=2$, with $F$ a single point, and $N$ drawn in blue.
\begin{center}
\includegraphics[width=9cm]{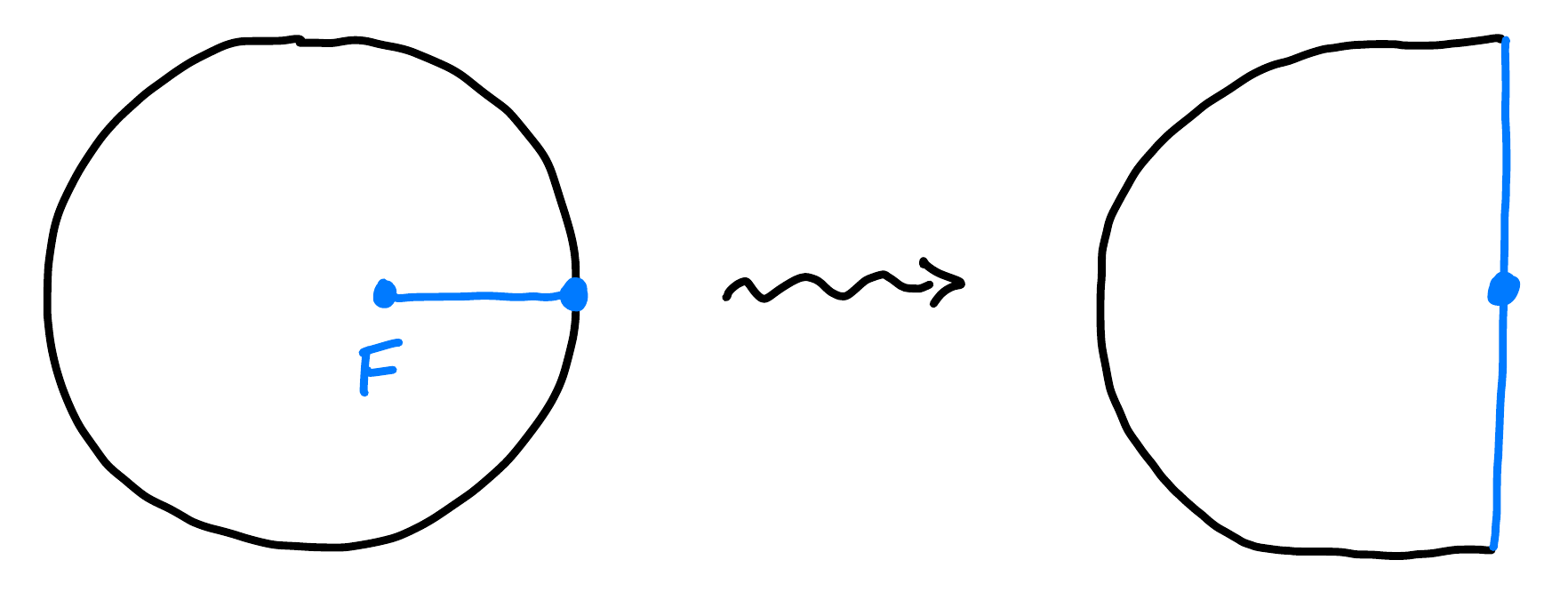} 
\end{center} 
The pair $(D^n,F)$ may be recovered from this cut-open ball by gluing $N$ to itself using the involution $\tau$ given by reflection on the fibres.  We obtain the double cover $\Sigma(D^n,F)$ of $D^n$ branched along $F$ by taking two copies of the cut-open $D^2$ and gluing them to each other via the same involution $\tau$.
\begin{center}
\includegraphics[width=5cm]{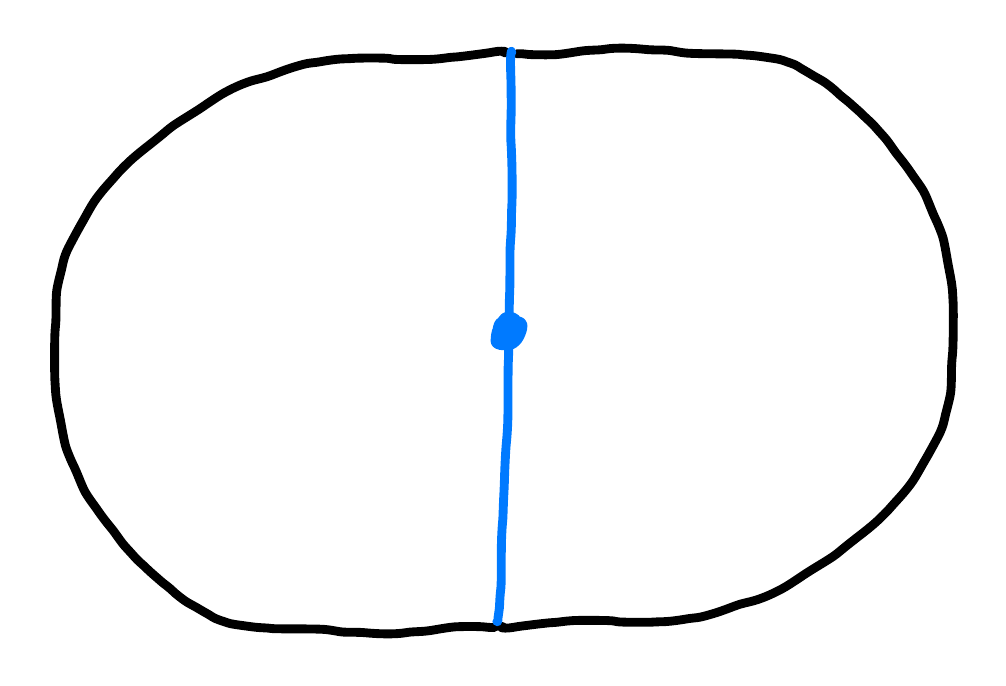} 
\end{center} 
We also note for later that if one glues the two copies of the cut-open $D^n$ by identifying copies of $N\times I$ instead of just $N$, one obtains the same manifold $\Sigma(D^n,F)$, up to diffeomorphism.  Here is a sketch for the $n=2$ example.
\begin{center}
\includegraphics[width=12cm]{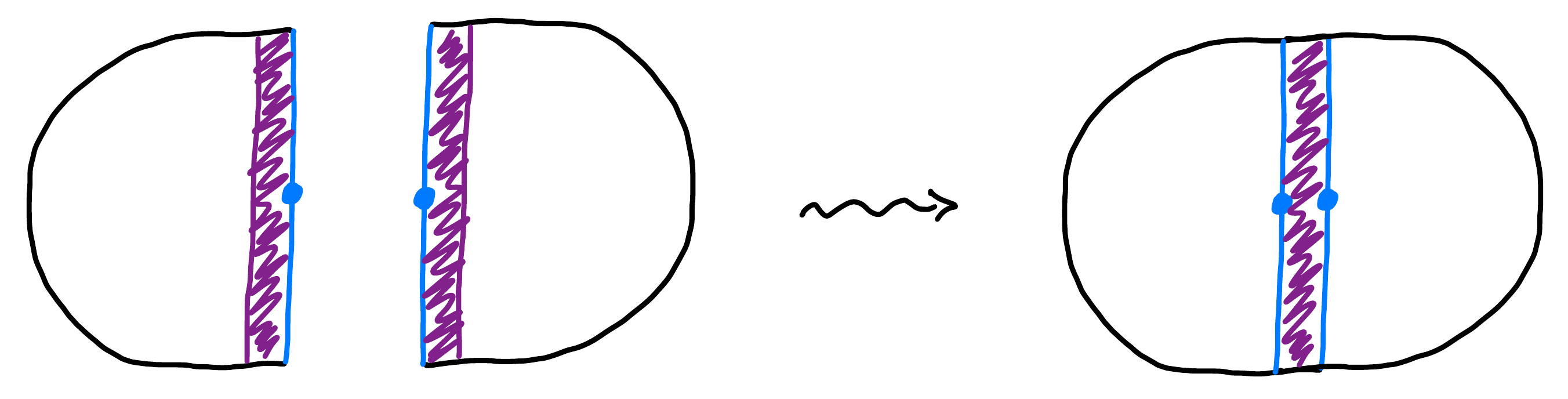} 
\end{center} 

\subsection*{The double branched cover of a link in $S^3$}
As a 4-dimensional example of the above method, we will explain how to derive a Kirby diagram and also a Heegaard diagram for the double branched cover of a link in $S^3$.  This Kirby diagram was first obtained by Ozsv\'{a}th and Szab\'{o}, and the closely related Heegaard diagram is due to Greene \cite{greenedbc,osdbc}; their proofs differed from that given here.  We begin with a diagram of our link $L$; let's choose a diagram of the figure-8 knot for an example, shown in Figure \ref{fig:fig8dbc} along with the resulting Kirby and Heegaard diagrams.

\begin{figure}[htbp]
\centering
\includegraphics[width=14cm]{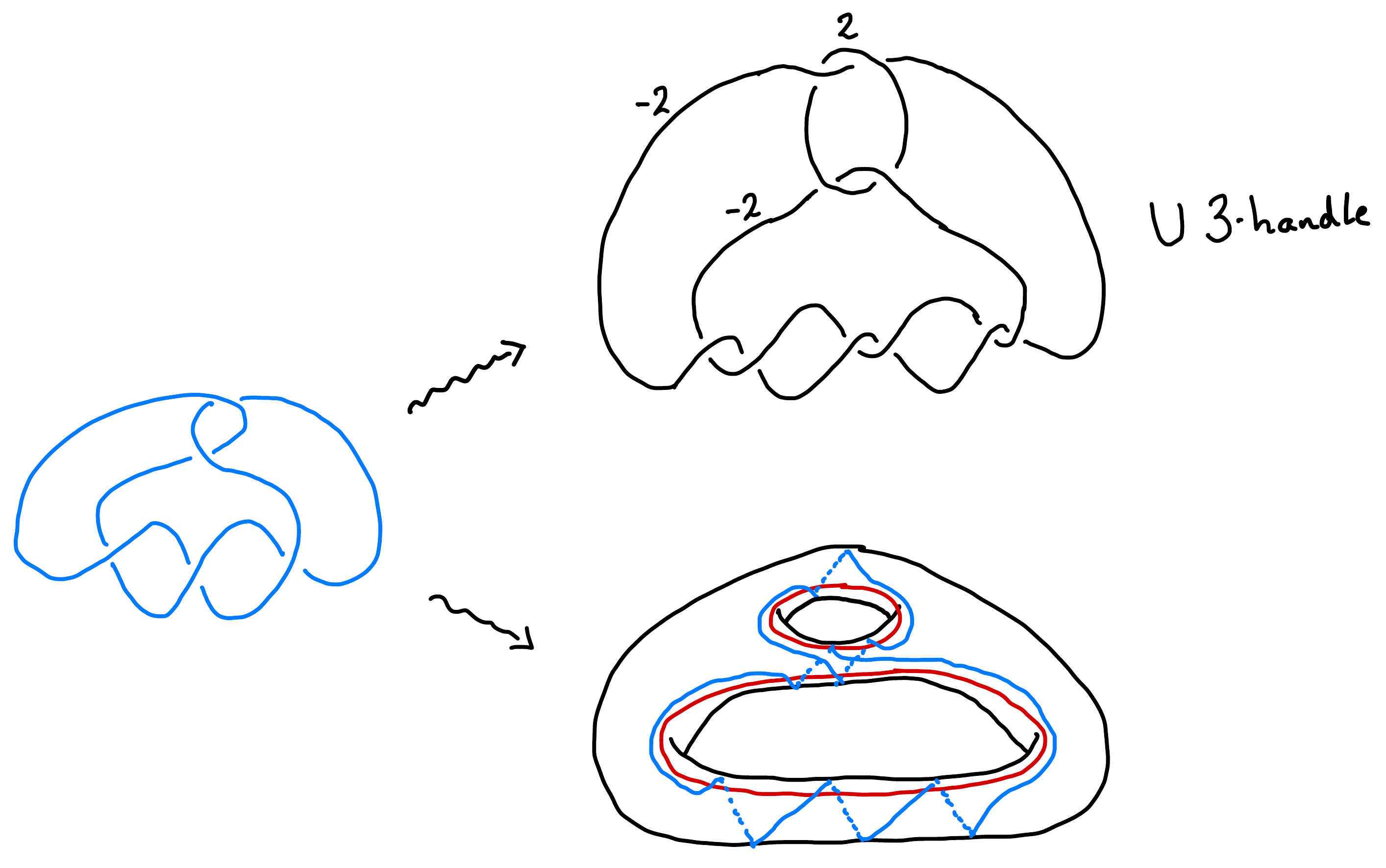} 
\caption
{{\bf Kirby and Heegaard diagrams for the double branched cover of the figure-8.}}
\label{fig:fig8dbc}
\end{figure}

To draw these,  we first choose a \emph{chessboard colouring} of the diagram: we colour the regions of the diagram black or white, so that at each crossing the colours meet like squares on a chessboard.  Here is one of the two choices for our example.
\begin{center}
\includegraphics[width=6cm]{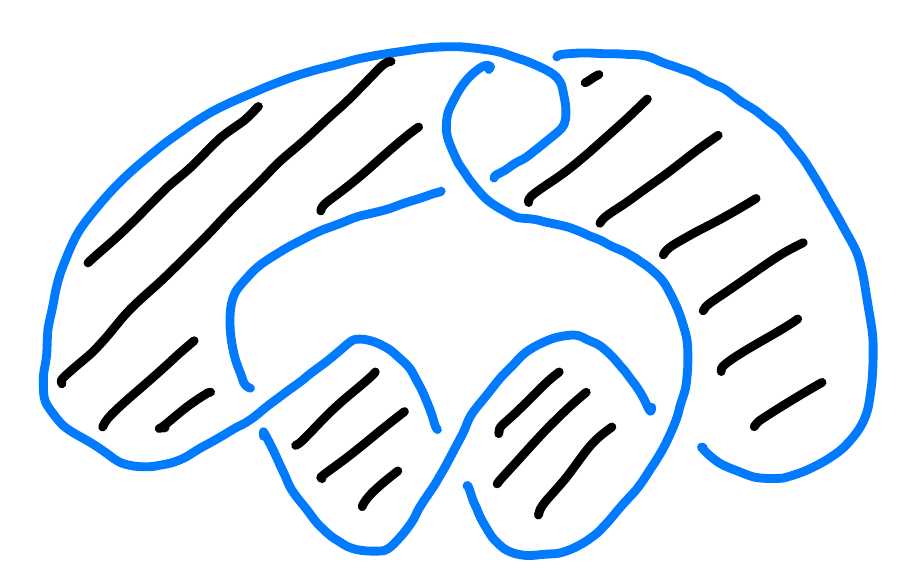} 
\end{center} 
Let $F$ be the resulting black spanning surface embedded in $S^3$.  We obtain the Kirby diagram from the link diagram by copying the diagram with the following substitutions at crossings (each crossing gets replaced by a clasp between white regions).
\begin{center}
\includegraphics[width=10cm]{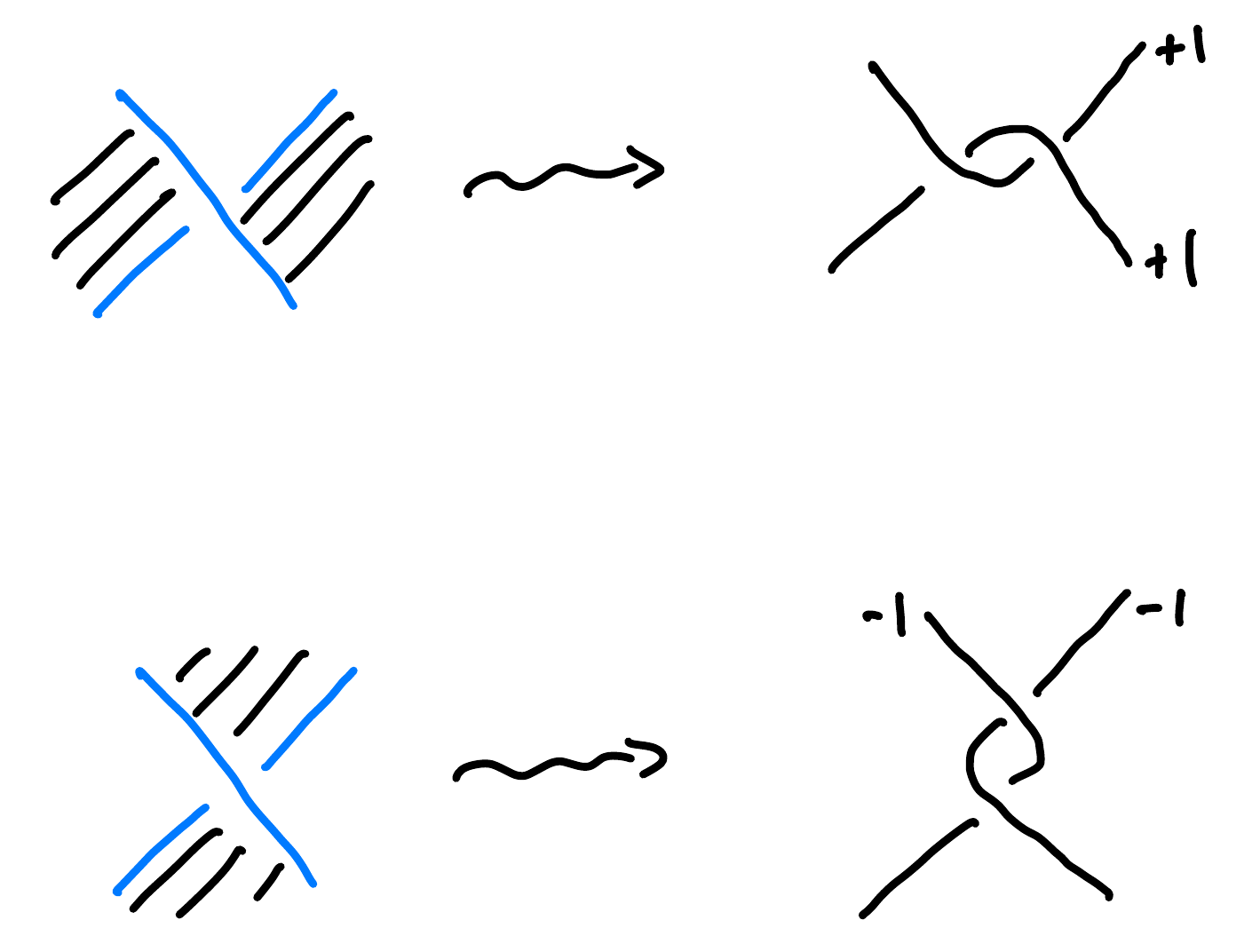} 
\end{center} 
This results in one 2-handle attaching circle for each white region of the diagram.  The framing coefficient of each is given by summing the local contributions ($+1$ or $-1$) from each crossing.  We then add one 3-handle; it turns out that this may be cancelled with any of the 2-handles in the diagram without changing the 4-manifold.  We will see that this is in fact a Kirby diagram for $\Sigma(D^4,F)$, where by an abuse of notation we also use $F$ to refer to the properly embedded surface obtained by pushing the interior of the black surface into the 4-ball.

The Heegaard diagram also makes use of the chessboard colouring and the black surface $F$.  For this we require $F$ to be connected.  Let $N$ again be a regular neighbourhood of $F$ which is the total space of an $I$-bundle over $F$.  The boundary of $N$ is our Heegaard surface $\Sigma$.  The red curves are exactly given by the boundaries of all but one of the white regions of the chessboard-coloured diagram; the omitted region can be chosen arbitrarily.  The blue curves are obtained by taking a parallel copy of the red curves, and then applying a Dehn twist for each crossing in the diagram to the one or two blue curves passing through that part of $\Sigma$ (as in Figure \ref{fig:fig8dbc}).

\begin{proposition} The Kirby diagram described above represents the double cover $\Sigma(D^4,F)$ of the 4-ball branched along the pushed-in black surface.  The 3-handle may be cancelled with any of the 2-handles in the diagram.  The Heegaard diagram represents the double cover $\Sigma(S^3,L)$ of $S^3$ branched along the link $L$.
\end{proposition}
\begin{proof}

We note that the $n$-ball $D^n$ has a handle decomposition given by a single $0$-handle and a cancelling pair consisting of an $(n-2)$-handle and an $(n-1)$-handle.  Equivalently we can build $D^n$ by starting with $S^{n-2}\times D^2$ and adding a single $(n-1)$-handle, whose attaching sphere is $S^{n-2}\times \{\mbox{point}\}$.  Explicitly, we take a regular neighbourhood of the equatorial $S^{n-2}$ --- the intersection of the boundary of $D^n$ with a hyperplane through the origin --- and then attach a thickened-up equatorial slice (a neighbourhood of a disk in the same hyperplane).  In particular we can obtain $D^4$ by starting with $S^2\times D^2$ and attaching a 3-handle along $S^2\times \{\mbox{point}\}$.

So what does this have to do with double branched covers?  Suppose we have a copy of $D^4$ with a link $L$ in its boundary.  We project $L$ to the equatorial $S^2$ in the boundary sphere to obtain our diagram, and we chessboard-colour it to obtain our black surface $F$.  We take $N$ to be a regular neighbourhood of $F$ in $S^3$ and $N\times I$ to be a regular neighbourhood of $F$ in $D^4$.  The intersection of $N\times I$ with the 2-sphere of the diagram contains the black regions, connected together at the crossings; its complement in this 2-sphere consists of one disk for each white region.
By attaching one 2-handle for each white region of the diagram to $N\times I$ we obtain a regular neighbourhood $S^2\times D^2$ of the sphere containing the diagram, and then we attach a 3-handle to complete the 4-ball.

Now consider the description of $\Sigma(D^4,F)$ given above, by gluing together two copies of $D^4$ cut open along the homotopy pushing $F$ into the interior.  Each of these has a copy of $N$ in its boundary, and we recall that $N$ is the total space of an $I$-bundle over $F$.  We use the involution $\tau$ coming from reflection on the fibres of this $I$-bundle to identify the copies of $N\times I$ in each cut-open $D^4$; the quotient is $\Sigma(D^4,F)$.  We want to describe a handle decomposition of this.  We take one of the cut-open 4-balls to be our 0-handle.  This also contains $N\times I$ from the second copy, via the gluing map.  We then add a 2-handle for each white region and a 3-handle to complete the second copy of the cut-open $D^4$.

To draw the Kirby diagram, we need to understand how these handles are attached.  To build the 4-ball starting with $N\times I$, the attaching circles for the 2-handles lie in the plane of the diagram and are 0-framed.  After attaching all but one of the 2-handles to $N\times I$ the result is a 4-ball $D^2\times D^2\subset S^2\times D^2$; then attaching the last 2-handle changes this to $S^2\times D^2$, but this is cancelled by the attachment of the 3-handle.
To build the double branched cover, we need to see the image of these attaching circles and framings under the gluing map $\tau$.  The local picture for this is shown in Figure \ref{fig:dbc2h}.

\begin{figure}[htbp]
\centering
\includegraphics[width=12cm]{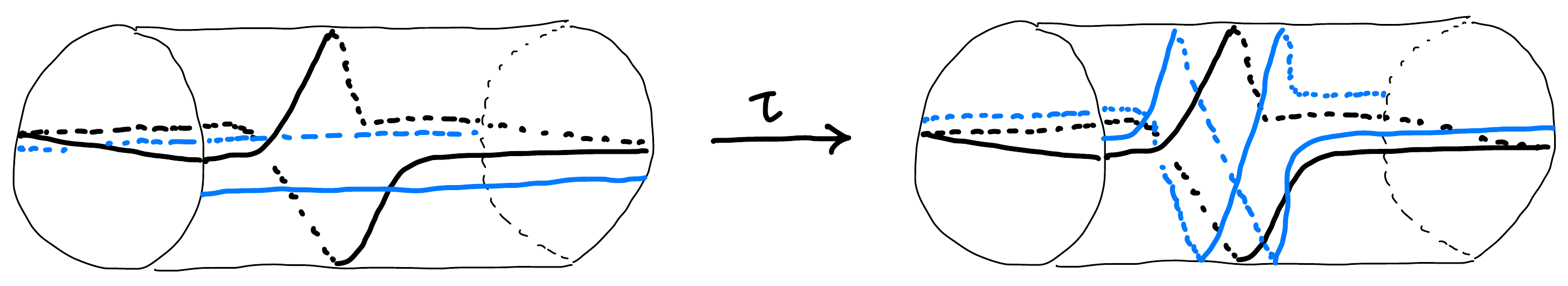} 
\caption
{{\bf The involution $\tau$ near a crossing.}}
\label{fig:dbc2h}
\end{figure}

We see that a crossing in the diagram results in a clasp between the 2-handle attaching circles coming from the white regions on either side of the crossing.  The effect on a framing curve, which is just a pushoff on the surface $\partial N$, is similar and results in  the framing curve getting a full twist around the attaching circle.  These local contributions exactly correspond to the description given above of the 2-handle attaching circles and their framing coefficients.

Finally the boundary $\Sigma(S^3,L)$ of $\Sigma(D^4,F)$ is the union of the boundaries of the two 4-balls that we glue together, minus the interior of $N$.  The complement of $N$ in $S^3$ is a 3-dimensional handlebody, and thus the construction naturally yields a Heegaard decomposition of $\Sigma(S^3,L)$.  The red curves of the Heegaard diagram  bound disks in the first copy of $S^3\setminus N$, and via Figure \ref{fig:dbc2h} we see that the blue curves bound in the second copy.  
\end{proof}

\subsection*{More double branched covers of $D^4$}
We have seen how to draw a simple Kirby diagram for the double cover of $D^4$ branched along the black surface of a link diagram.  One can also draw Kirby diagrams of more general surfaces in $D^4$ and hence compute their homological invariants including intersection form.  For details on this see \cite{akbulut,AK,GS,GLslice}.  
We give a brief description here for the case of a ribbon surface.  We take $F$ to be a ribbon-immersed surface in $S^3$ with a handle decomposition chosen so that all ribbon singularities consist of 1-handles passing through 0-handles.  Draw this in $\rr^3$ with the 0-handles in the $xz$-plane, lying just below the $x$-axis.  
All attachments of 1-handles to 0-handles should be on the $x$-axis and all ribbon singularities below the $x$-axis in the interiors of the 0-handles.  It helps to also ``comb the 1-handles up" so that they mostly lie near the $xz$-plane and above the $x$-axis.  We then take the preimage of this diagram under the double branched covering map
$$\phi:\rr^3\to\rr^3,\quad(x,\zeta=y+iz)\mapsto(x,\zeta^2).$$
An example is shown in Figure \ref{fig:stevedoredbc} for the stevedore ribbon disk encountered in Lecture 1.  
In practice one can draw the preimage of the diagram under $\phi$ by taking the diagram in the upper half plane, away from the $x$-axis, and revolving it around the $x$-axis, then connecting the tangle diagrams in the upper half plane and the lower half plane with a half twist of 4 strands for each ribbon singularity, and with no twist on two strands at the end of each 1-handle (see Figure \ref{fig:stevedoredbc}).

\begin{figure}[htbp]
\centering
\includegraphics[width=12cm]{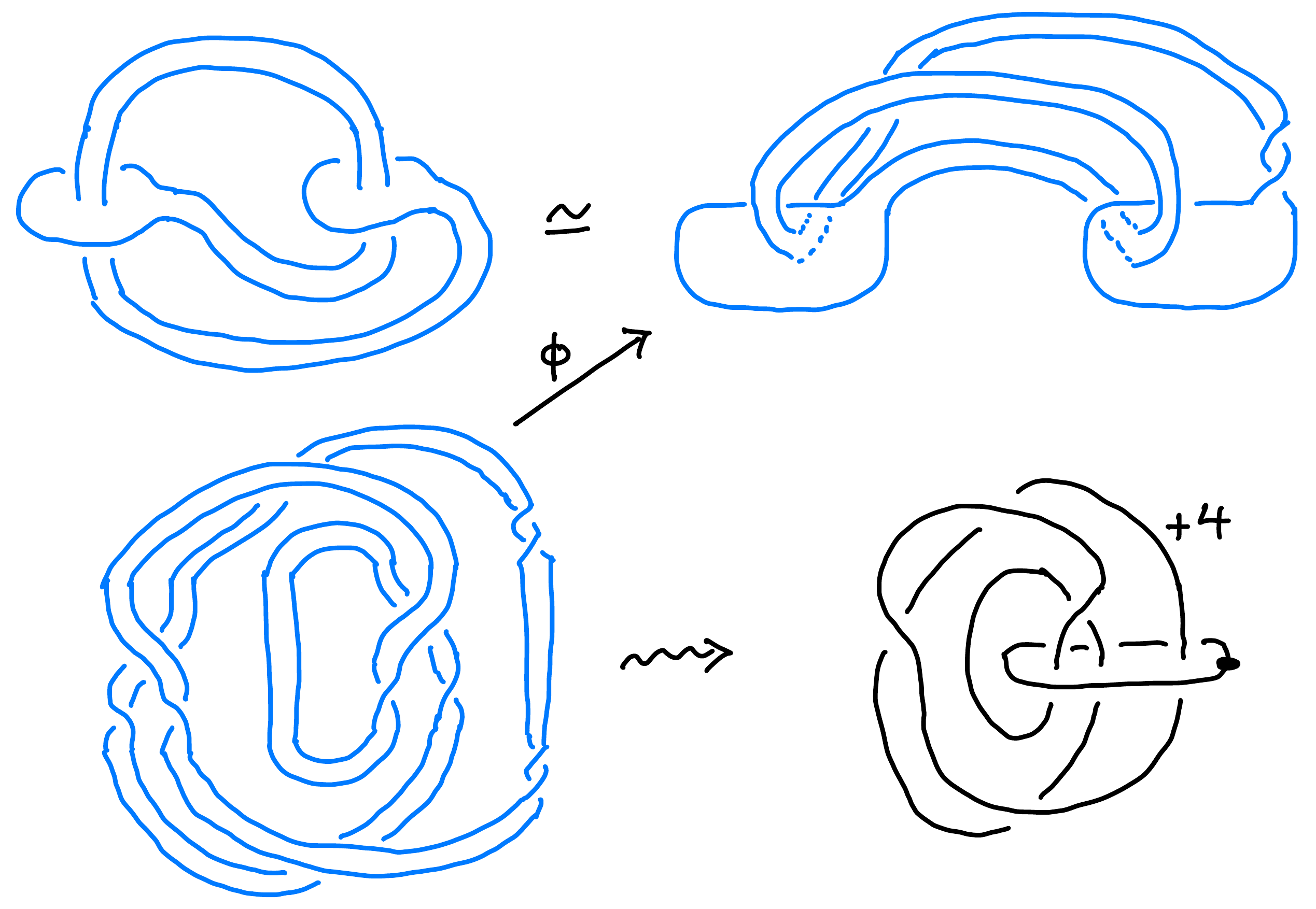} 
\caption
{{\bf The double branched cover of the stevedore ribbon disk.}}
\label{fig:stevedoredbc}
\end{figure}

From this we draw a Kirby diagram as follows.  Each 1-handle of $F$ gives rise to an embedded annulus in $\phi^{-1}(\rr^3)$; we take the core of this annulus to be the attaching circle for a 2-handle, and the linking number of the boundary of the annulus is the framing coefficient.  We also draw 1-handles in dotted circle notation, one for all but one of the 0-handles of $F$; these are drawn around the boundary of the preimage under $\phi$ of the corresponding 0-handle of $F$.  Again, see Figure \ref{fig:stevedoredbc} for an example.  For details and more examples see the sources cited above.  

The main point here is that given a surface $F$ in $D^4$, there are known methods for obtaining a Kirby diagram of $\Sigma(D^4,F)$ and hence computing the homology groups and the intersection pairing of this manifold.  In \cite{GLslice} the author and Strle give a definition of \emph{disoriented homology} groups associated to such a surface $F$, together with a bilinear pairing on one of the groups.  We  show that these are isomorphic to the homology groups of the double branched cover, together with the intersection pairing.  This enables these invariants to be computed directly from a diagram of $F$, without needing to first obtain a Kirby diagram.

For the remainder of these notes, if $F$ is a properly embedded surface in $D^4$ or a surface embedded or ribbon-immersed in $S^3$ giving rise to a pushed-in surface in $D^4$, then we will use the notation $\Sigma(D^4,F)$ to refer to the resulting double branched cover.  We will also use $\Lambda_F$ to refer to the intersection lattice of this manifold, or in other words the second homology group modulo torsion of $\Sigma(D^4,F)$ together with its intersection pairing.


\section{Lecture 3}

Our goal in this final lecture is to explain how Donaldson's diagonalisation theorem \cite{thmA} gives rise to a very useful slice obstruction for alternating knots, and also to an obstruction to any given band move being part of a sequence that yields a ribbon disk.  A computer search for ribbon alternating knots based on this has been implemented in joint work of the author and Frank Swenton \cite{ribbonalg} which has resulted in a database of over 200,000 previously unknown slice knots.

\subsection*{Goeritz, Gordon-Litherland, and Donaldson}

We begin by describing the Goeritz lattice of a connected chessboard-coloured diagram.  Label the white regions $R_0,\dots,R_m$ and then write down a square matrix $\widehat{G}$ as follows.  The rows and columns of $\widehat{G}$ are numbered from $0$ to $m$.  The diagonal entries $\widehat{G}_{ii}$ are given by a signed count of the crossings around region $R_i$, with signs given as follows.
\begin{center}
\includegraphics[width=6cm]{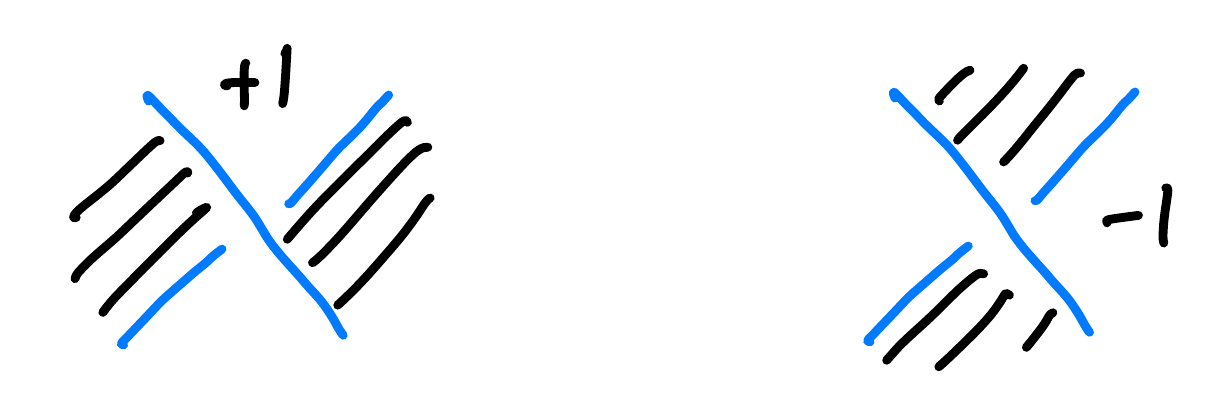} 
\end{center}
The off-diagonal terms $\widehat{G}_{ij}$ are given by a signed count of the crossings between $R_i$ and $R_j$, with the opposite signs to those above.  The Goeritz matrix $G$ is obtained from $\widehat{G}$ by deleting the $0$th row and column.  Here is an example.
\begin{center}
\includegraphics[width=8cm]{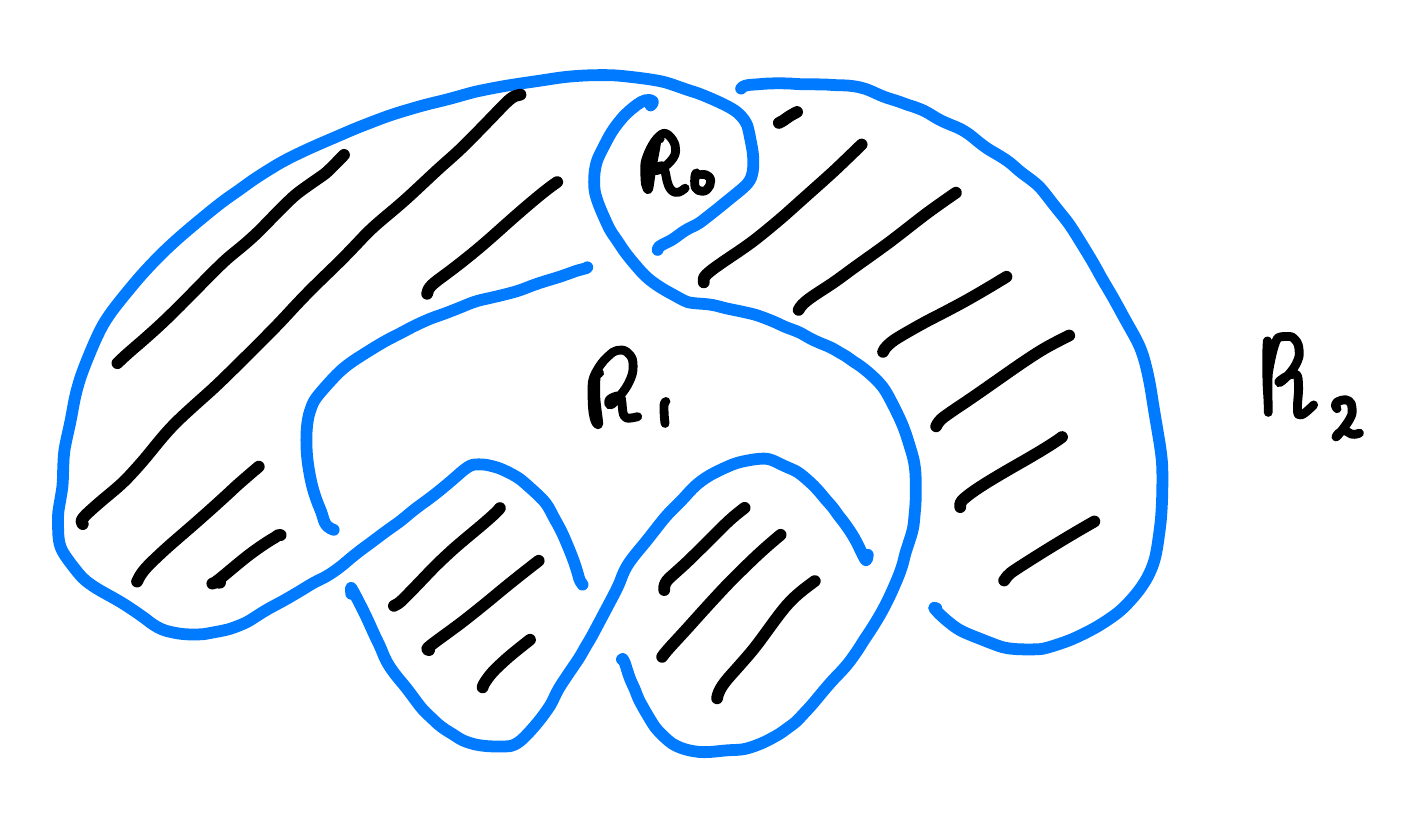} 
\end{center}
We get
$$\widehat{G}=
\begin{bmatrix} 2 & -1 & -1\\ -1& -2 & 3\\ -1 & 3 & -2
\end{bmatrix},
\mbox{ and }
G=\begin{bmatrix} -2 & 3\\ 3 & -2
\end{bmatrix}.
$$
This was introduced by Goeritz in 1933 \cite{goeritz}.  He showed that the determinant and more generally the \emph{Minkowski units} \cite{mur} of $G$ give rise to knot invariants.  The astute reader may notice that the matrices above represent the intersection pairing which may be computed from the Kirby diagram for $\Sigma(D^4,F)$ in Figure \ref{fig:fig8dbc}.

The Goeritz lattice of a chessboard-coloured diagram is described as follows.  One takes the (free) abelian group $\Lambda$ with one generator given by each white region, and one relation given by the sum of all of the generators.  Note this has basis given by all but one of the generators.  The pairing $$\lambda:\Lambda\times \Lambda\to \zz$$ is given by the matrix $\widehat{G}$.
Comparing to the Kirby diagram from Lecture 2, we recover Gordon-Litherland's result that the Goeritz lattice is the intersection lattice on the second homology group of the double branched cover $\Sigma(D^4,F)$ \cite{GL}.

We also observe the following.  The author learned this proof from Josh Greene.  For more topological proofs see \cite[Prop. 4.1]{altsurf} or \cite[Lemma 3.1]{ribbonalg}.
\begin{lemma}
\label{lem:altdef}
The Goeritz lattice of a connected alternating diagram is definite.  More precisely, for one choice of chessboard-colouring the Goeritz lattice is positive definite,  and for the other it is negative definite. 
\end{lemma}
\begin{proof} The Goeritz lattice of the crossingless diagram of the unknot is trivial and hence both positive and negative definite.  Assume now that our diagram has at least one crossing.
Being alternating is equivalent to all crossings having the same sign when computing the Goeritz matrix.  Assume that all of these signs are $+1$, so that the diagonal entries of $\widehat{G}$ are all positive and the off-diagonal entries are all nonpositive.  The rows of $\widehat{G}$ sum to zero (for any diagram).  Let $r_i$ denote the sum of the absolute values of the off-diagonal entries of $G$ in the $i$th row.  Then $r_i=G_{ii}-|\widehat{G}_{0i}|$ is less than or equal to $G_{ii}$.  The eigenvalues of the symmetric matrix $G$ are real and by Gershgorin's circle theorem \cite{circlethm}, they are contained in
$$\bigcup[G_{ii}-r_i,G_{ii}+r_i]$$
and hence nonnegative.  Since their product is the determinant of a link  with connected alternating diagram and hence nonzero \cite{crowell}, the eigenvalues are all positive.  The proof for the other colour is similar.
\end{proof}

We note that Josh Greene and Josh Howie independently proved in 2015 that the property in \Cref{lem:altdef} characterises alternating links: a link admits a connected alternating diagram if and only if it bounds two embedded surfaces in $S^3$ for which the double branched cover of $D^4$ has definite intersection pairing (one positive and one negative) \cite{altsurf,howie}.

We now describe a key slice obstruction for alternating knots.  The double cover of $S^3$ branched along the unknot is $S^3$, and the double cover of $D^4$ branched along an unknotted properly embedded $D^2$ --- or in other words along the standard slice disk for the unknot --- is again $D^4$.  An imprecise slogan that has some validity is that ``rational homology does not see knotting".  This manifests itself here in the fact that the double cover of $S^3$ branched along any knot is a rational homology $S^3$ \cite[Chapter 9]{lick}, and the following fact due to Casson-Gordon \cite{CG}: the double cover $\Sigma(D^4,\Delta)$ branched along any properly embedded disk has the same rational homology as the 4-ball.

We can combine this with Donaldson's celebrated diagonalisation theorem \cite{thmA}, which states that if $X$ is a closed smooth 4-manifold with a positive-definite intersection form then in fact its intersection lattice $\Lambda_X=(H_2(X;\zz)/\tors,Q_X)$ is isomorphic to the standard $\zz^n$ lattice; in other words, $\Lambda_X$ admits an orthonormal basis.  This gives us the following.

\begin{lemma}
\label{lem:don}
Let $K$ be an alternating knot and let $F$ be the black surface of its alternating diagram, with chessboard colouring chosen so that $\Lambda_F$ is positive definite.  Let $b$ be a band move applied to $K$.  Then
\begin{enumerate}
\item if $K$ is slice, then $\Lambda_F$ embeds as a finite index sublattice of a standard $\zz^m$ lattice;
\item if $b$ is part of a sequence of band moves giving a ribbon disk for $K$, then there is a commutative diagram of lattice morphisms
$$\begin{tikzpicture}[scale=0.5]
\node at (0,0) {$\Lambda_F$};
\node at (5,0) {$\zz^m$,};
\node at (2.5,-2) {$\Lambda_{F\cup b}$};
\draw[->] (1,0) to (4,0);
\draw[->] (0.7,-0.7) to (1.6,-1.4);
\draw[->] (3.4,-1.4) to (4.3,-0.7);
\end{tikzpicture}$$
where the horizontal arrow is the inclusion from (1), and $F\cup b$ is a ribbon-immersed surface in $S^3$ obtained by attaching the band $b$ to $F$.
\end{enumerate}
\end{lemma}
\begin{proof}[Proof sketch] For part (1), we use the fact that the double branched cover $\Sigma(D^4,F)$ of $F$ is positive definite and the double branched cover $\Sigma(D^4,\Delta)$ of a slice disk is a rational homology ball.  Gluing these along their common boundary gives a smooth closed manifold $X=\Sigma(S^4,F\cup\Delta)$; this is the double cover of the 4-sphere branched along a properly embedded closed surface obtained from $F$ and $\Delta$.  The Mayer-Vietoris sequence shows that $\Lambda_F$ is a finite index sublattice of the intersection lattice of $X$.  It follows that the latter  is also positive definite, and thus  diagonalisable by Donaldson.  (This is the slice obstruction used by Lisca in \cite{lisca}.)

Part (2) follows from the fact that the closed manifold $X=\Sigma(S^4,F\cup\Delta)$ is built from $\Sigma(D^4,F)$ by adding a 2-handle for each 1-handle of $\Delta$ followed by a 3-handle for each minimum of $\Delta$ and finally a 4-handle (see for example \cite{GS,GLslice}).  This gives a commutative diagram of manifold inclusions
$$\begin{tikzpicture}[scale=0.5]
\node at (-1.2,0) {$\Sigma(D^4,F)$};
\node at (5,0) {$X$,};
\node at (2.5,-2) {$\Sigma(D^4,F\cup b)$};
\draw[->] (1,0) to (4,0);
\draw[->] (0.7,-0.7) to (1.6,-1.4);
\draw[->] (3.4,-1.4) to (4.3,-0.7);
\end{tikzpicture}$$
from which the commutative diagram of lattice morphisms follows.
\end{proof}

There are a couple of points to note about \Cref{lem:don}.  The first is that the mirror of a slice knot is slice, so one may change the crossings and the colour and apply the  lemma twice to each alternating knot and to each band move.  The second is that following Greene and Jabuka \cite{GJ}, one may strengthen the lemma using Heegaard Floer homology.  This puts extra conditions on the lattice embedding in the lemma, namely that every equivalence class in $\zz^m/\Lambda_F$ is represented by an element of $\{0,1\}^m$.

As an example let's consider the stevedore knot $K$ together with the following band move $b$.
\begin{center}
\includegraphics[width=8cm]{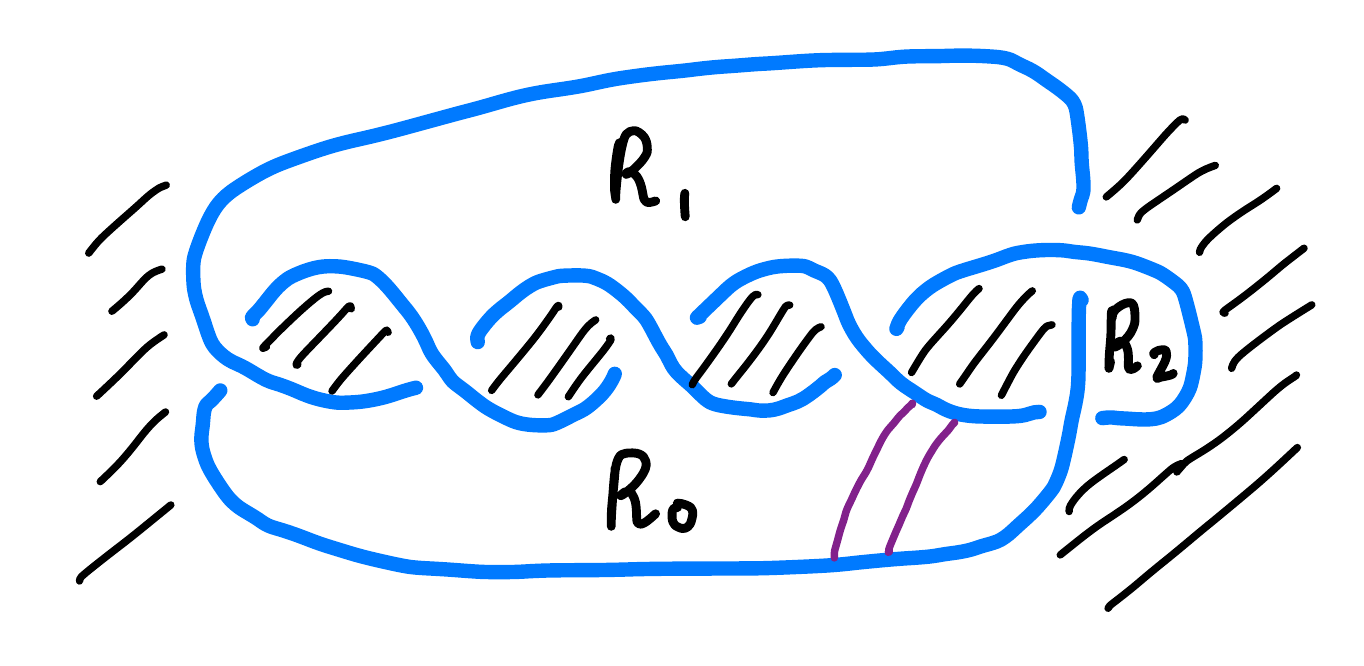} 
\end{center}
We can calculate $\Lambda_F$ using the Goeritz matrix: we find this has rank 2 and the pairing is given by $G=\begin{bmatrix}5&-1\\-1&2\end{bmatrix}$, using generators $v_1$ and $v_2$ corresponding to the white regions $R_1$ and $R_2$.  The first part of \Cref{lem:don} tells us that this admits a lattice embedding in $\zz^2$ with the dot product: sure enough if we map $v_1$ to $(2,1)$ and $v_2$ to $(-1,1)$, this preserves the pairing.  Moreover, this embedding is unique up to symmetries.  Now we apply the second part of the lemma.  The band $b$, shown in purple, would become part of the black surface for the new diagram.  This would split $R_0$ into two new regions.  We can choose the new generator $v_3$ to correspond to the resulting small white region on the right side of the band.  Then we can calculate the Goeritz of the new diagram and we find that
\begin{align*}
v_3\cdot v_1&=0\\
v_3\cdot v_2&=-1.
\end{align*}
By part (2) of the lemma, if $b$ is part of a sequence giving a ribbon disk, then there exists a vector $(x,y)\in\zz^2$ with
\begin{align*}
(x,y)\cdot (2,1)&=0\\
(x,y)\cdot (-1,1)&=-1.
\end{align*}
Solving  over the rationals we find $(x,y)=(1/3,-2/3)\notin\zz^2$.  We conclude that this band move is not part of any sequence giving a ribbon disk for $K$.

The reader may compare this to the band move given below, and see that it converts $K$ to the two-component unlink, and in fact is obtained by reversing the movie shown for $K$ in Lecture 1.  In this case the new white region to the right of the band gets sent to $(0,-1)\in\zz^2$.
\begin{center}
\includegraphics[width=\textwidth]{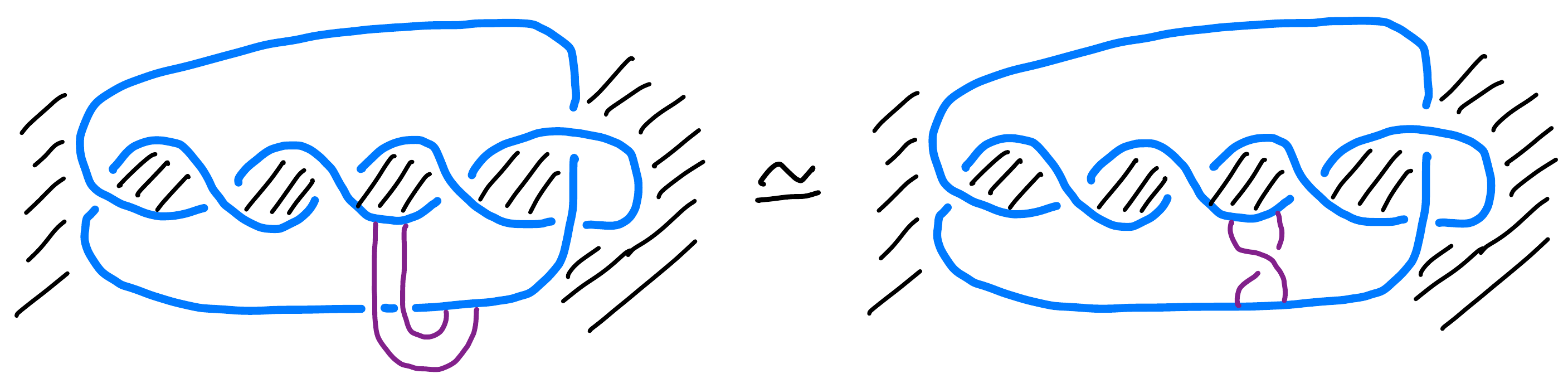} 
\end{center}

For more complicated band moves one cannot always compute $\Lambda_{F\cup b}$ using the Goeritz matrix.  One may always either draw a Kirby diagram, or calculate using disoriented homology \cite{GLslice}.  This helpful tool does not seem to be sufficient to give an answer to the following.

\begin{question}
Is there a unique isotopy class of ribbon disk for the stevedore knot?
\end{question}

The search algorithm of \cite{ribbonalg} is based on the band obstruction of \Cref{lem:don}, and on an optimistic notion that alternating knots like to stay alternating.  Recall that the nullity of a link in $S^3$ is the first Betti number of its double branched cover.  The nullity of a knot is zero, and that of an unlink is one less than the number of components.  In a movie of a ribbon disk starting with a knot and ending with an unlink, each band move increases the nullity  by one.  Connected alternating diagrams have nullity zero, and in fact the nullity gives a lower bound on how many ``nonalternating" crossings a diagram can have.  Our algorithm then applies the following steps to a given alternating knot:
\begin{enumerate}
\item Check if the knot is obstructed by \Cref{lem:don} (1) from being slice.
\item Try all band moves which are unobstructed by \Cref{lem:don} (2) and which keep the resulting link as close to alternating as possible (i.e., the number of nonalternating crossings and the nullity both increase by 1).
\item For each such band move, simplify the resulting link diagram using Tsukamoto-type moves.
\item Repeat the previous two steps until all possibilities are exhausted or a crossingless diagram is obtained.  In the latter case, we declare the knot to be \emph{algorithmically ribbon}.
\end{enumerate}

Some remarks are in order.  We show in \cite{ribbonalg} that the bands to be considered in step (2) take the following form,
\begin{center}
\includegraphics[width=10cm]{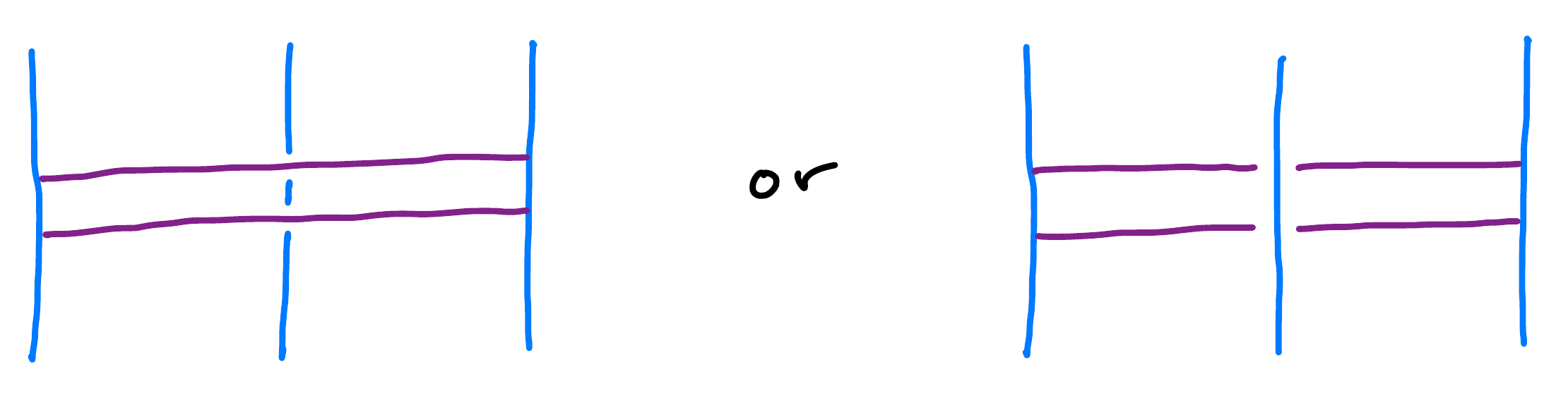} 
\end{center}
so that there are always finitely many to consider.  There are also finitely many alternating diagrams of a given knot, and we have to deal with all of them.

Also, what are Tsukamoto-type moves?  These are a set of  local modifications on diagrams which preserve isotopy type, and also the property of being close to alternating and unobstructed by a slight generalisation of \Cref{lem:don} which applies to such diagrams \cite[Prop. 3.5]{ribbonalg}.  They include Reidemeister I and II moves and five others shown below, the first of which is the famous flype move of P.~G.~Tait.
\begin{center}
\includegraphics[width=\textwidth]{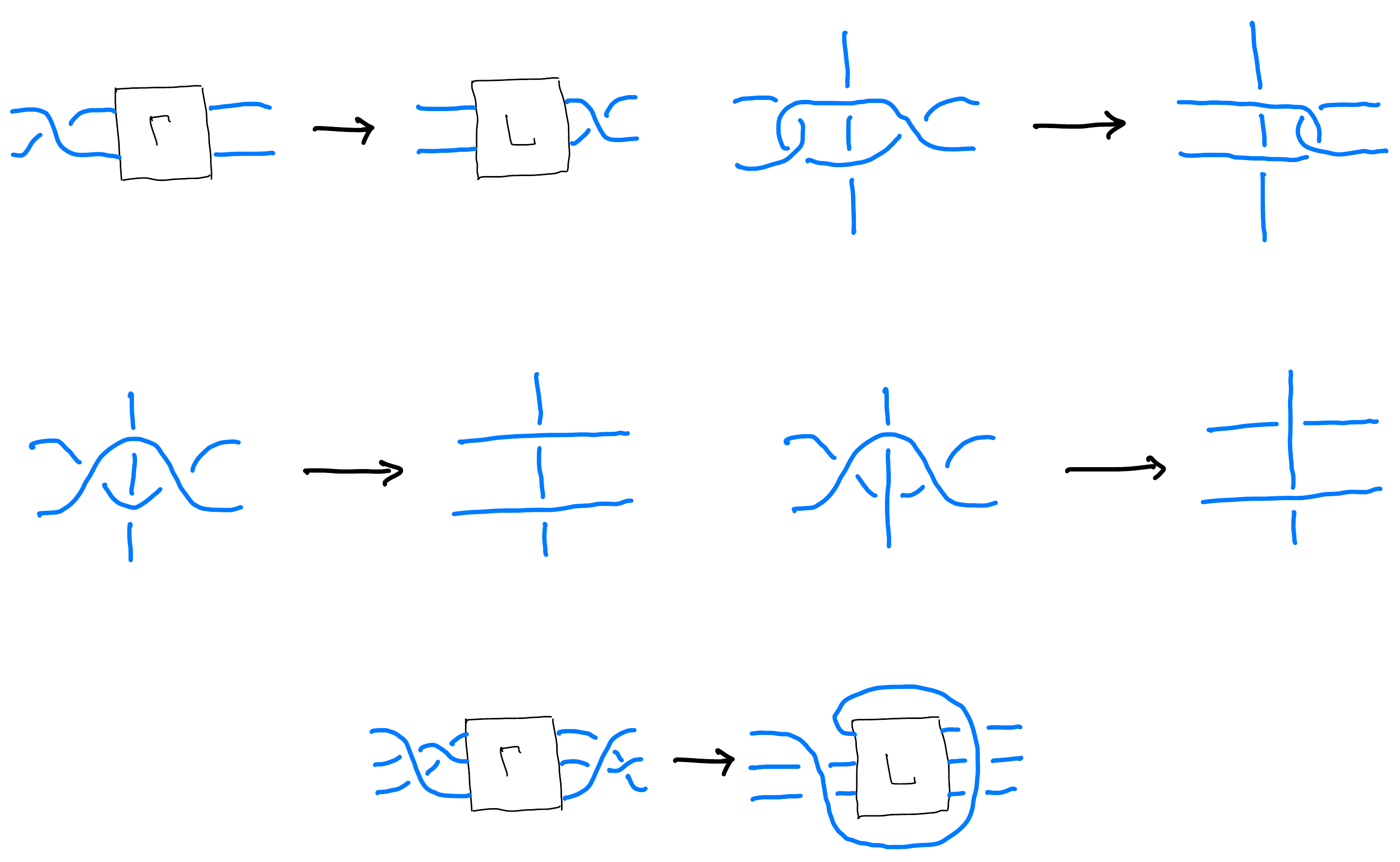} 
\end{center}
The third of these was called the \emph{untongue move} by Tsukamoto, who proved in \cite{tsuk} that every almost-alternating\footnote{A diagram is almost alternating if it contains exactly one nonalternating crossing.} diagram of the two-component unlink may be converted to the crossingless diagram by a sequence of flypes, untongues, and Reidemeister I and II moves.  It follows that if an alternating knot admits a ribbon disk which can be presented by a single band move converting it to an almost-alternating diagram of the unlink, then in fact this disk will be detected by our algorithm.  This together with experimental evidence and optimism led us to state the following in \cite{ribbonalg}.

\begin{conjecture}
If an alternating knot admits a ribbon disk with three critical points, then in fact it is algorithmically ribbon.
\end{conjecture}
If true, this would give an effective classification of such alternating knots:  Tsukamoto's theorem gives a recursive construction of all almost-alternating diagrams of the two-component unlink, and each of these would arise from four possible band moves from alternating diagrams as shown below.  This set of alternating diagrams would  contain repetitions and would include diagrams of links as well as knots.
\begin{center}
\includegraphics[width=\textwidth]{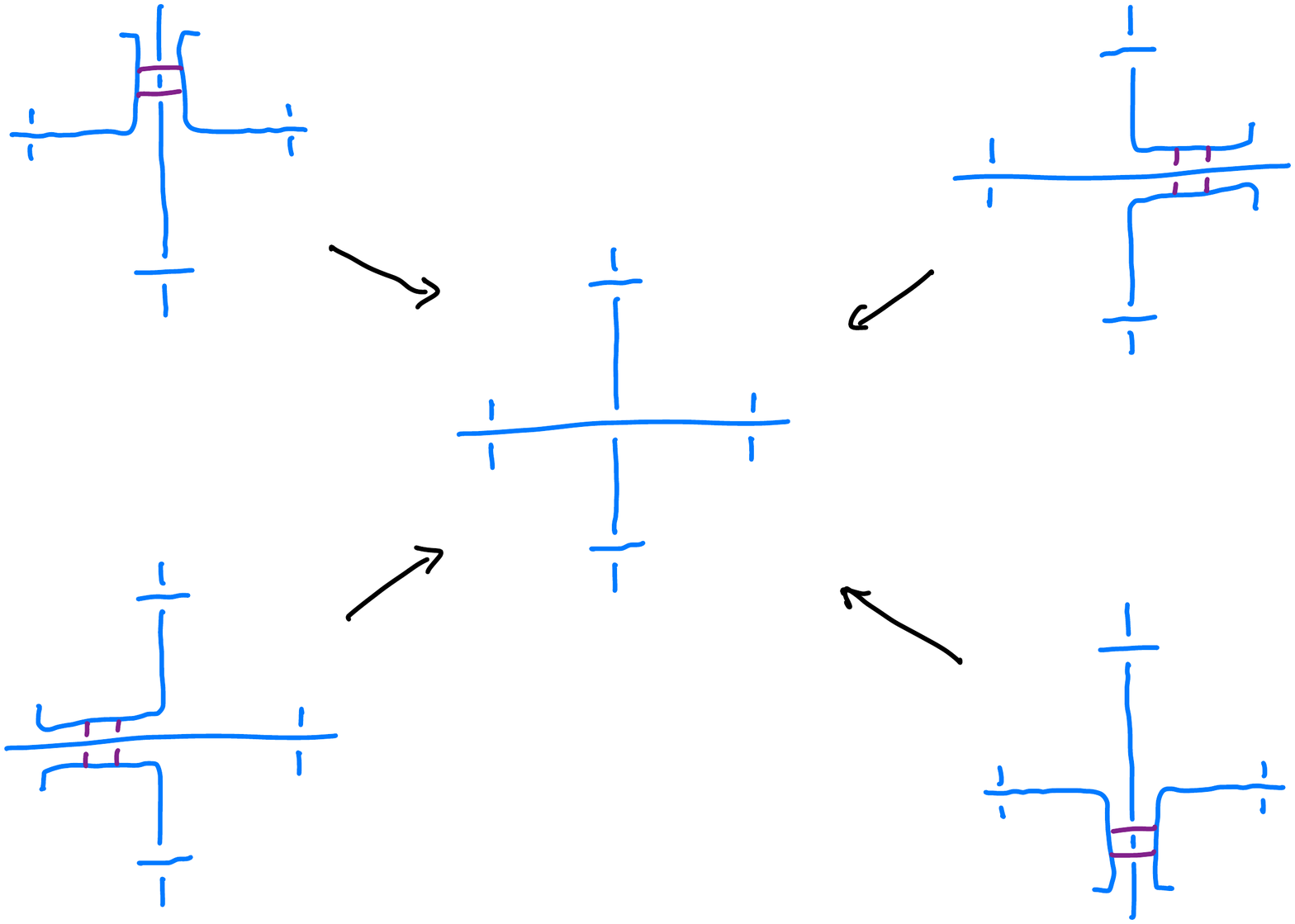} 
\end{center}

The algorithmic search for alternating knots may be considered a work in progress.  We have applied it to all 250,912,342 prime alternating knots of up to 20 crossings, using the prime alternating knot generator of Flint, Rankin, and de Vries \cite{PAKG}.  Of these, we find 438,375 are unobstructed from being slice by \Cref{lem:don} and 231,968 turn out to be algorithmically ribbon.  We also know of a further 628 examples of ribbon alternating knots of up to 20 crossings which are \emph{not} algorithmically ribbon.  The first of these was 12a631, shown to be ribbon by Seeliger \cite{seeliger} using symmetric union diagrams (see also \cite{knotinfo,lamm}).

How can we improve the algorithm?  One straightforward way is to implement more obstructions, including Levine-Tristram signatures and the Fox-Milnor condition (see for example \cite{liv2}).  Beyond this the hope is that we may understand better how ribbon disks can arise and expand the algorithm to find more of them.  Thus we will expand the set of knots on which our algorithm is able to determine the answer to the question ``Slice: yes or no?" and hopefully one day it will include all alternating knots.

For more details on the algorithm and to download Swenton's software package KLO in which it is implemented, we direct the reader to 
\begin{center}
\url{www.klo-software.net/ribbondisks}.
\end{center}

\subsection*{Acknowledgements}
I am very grateful to the organisers of Winter Braids IX for a wonderful conference and for their patience in waiting for these notes.  I also thank the participants of the school for being an encouraging and helpful audience.  Thanks to Sa\v{s}o Strle for helpful comments on an earlier draft, and to the anonymous referee for many helpful suggestions to improve the exposition.  I thank Alex Zupan and Jeff Meier for pointing out that a claimed positive answer to \Cref{q:unknot} in a previous version of these notes was unjustifed.


\bibliographystyle{amsplain}
\bibliography{WB}

\end{document}